\documentclass[11pt]{amsart}
\usepackage[margin=1in]{geometry}
\usepackage{fourier}
\usepackage{graphicx}
\usepackage{setspace}
\usepackage[T1]{fontenc}
\usepackage{color}
\usepackage{amsmath}
\usepackage{amssymb}
\usepackage{amsthm}
\usepackage{mathtools}
\usepackage{commath}
\usepackage[authoryear, round]{natbib}
\usepackage[bottom]{footmisc}

\newcommand{\Var}{\mathrm{Var}}
\newcommand{\Cov}{\text{Cov}}

\newcommand{\R}{\mathbb{R}}

\def\P{\mathbb{P}}
\def\E{\mathbb{E}}

\newcommand{\cconest}{C^n_{i_1, j_1}(s_1,t_1)}
\newcommand{\cctwost}{C^n_{i_2, j_2}(s_2,t_2)}

\newcommand{\N}{\mathbb{N}}
\newcommand{\im}{\mathrm{im}}

\newcommand{\X}{\mathcal{X}}

\newcommand{\ind}[1]{\mathbf{1}\big\{#1\big\}}

\DeclarePairedDelimiterX{\inp}[2]{\langle}{\rangle}{#1, #2}

\theoremstyle{plain}
\newtheorem{theorem}{Theorem}[section]
\newtheorem{lemma}[theorem]{Lemma}
\newtheorem{proposition}[theorem]{Proposition}
\newtheorem{corollary}[theorem]{Corollary}
\theoremstyle{definition}
\newtheorem{definition}[theorem]{Definition}
\newtheorem{remark}[theorem]{Remark}
\newtheorem{example}[theorem]{Example}

\title[Persistent homology of stationary processes]{Convergence of persistence diagrams \\ for discrete time stationary processes}
\author{Andrew M. Thomas}
\address{Department of Statistics and Actuarial Science, University of Iowa}
\thanks{I would like to thank two anonymous reviewers for their comments which greatly improved the presentation of this article. I would also like to thank the associate editor who noticed an issue with the proof of the CLT. A portion of this work was completed while the author was a postdoctoral associate at Cornell University. This work was funded in part by NSF grants DMS-2114143 and OAC-1940124.} 
\email{andrew-thomas@uiowa.edu}
\date{}

\begin{document}

\begin{abstract}
In this article we establish two fundamental results for the sublevel set persistent homology for stationary processes indexed by the positive integers. The first is a strong law of large numbers for the persistence diagram (treated as a measure ``above the diagonal'' in the extended plane) evaluated on a large class of sets and functions, beyond continuous functions with compact support. We prove this result subject to only minor conditions that the sequence is ergodic and the tails of the marginals are not too heavy. The second result is a central limit theorem for the persistence diagram evaluated on the class of all step functions; this result holds as long as a $\rho$-mixing criterion is satisfied and the distributions of the partial maxima do not decay too slowly. Our results greatly expand those extant in the literature to allow for more fruitful use in statistical applications, beyond idealized settings. Examples of distributions and functions for which the limit theory holds are provided throughout.
\end{abstract}

\maketitle

\vspace{-11pt}

\section{Introduction}

Understanding the persistent homology of large samples from various probability distributions is of increasing utility in goodness-of-fit testing \citep{biscio2020, krebs_fclt}. For goodness-of-fit testing in the ``geometric'' setting there are a number of results to choose from, as much attention has been focused on the limiting stochastic behavior of {\v C}ech and Vietoris-Rips persistent homology of (Euclidean) point clouds \citep[ibid. as well as][]{hiraoka2018, divol_polonik, krebs_polonik, owada2020, krebs2021, owada2022, bobrowski2024}. However, less attention has been focused on the asymptotics of the entire sublevel (or superlevel) set persistent homology of stochastic processes and random fields---with a few notable exceptions \citep{chazal2018, baryshnikov2019, miyanaga_diss, perez2022, hiraoka2022large}.  

In recent years, summaries of sublevel set persistent homology of time series---such as those we establish limit theory for below---have been applied to the problems of heart rate variability analysis \citep{chung2021, graff2021}, eating behavior detection \citep{chung_eat}, and sleep stage scoring using respiratory signals \citep{chung2024}. Thus, a comprehensive treatment of the asymptotic properties of sublevel set persistent homology of stochastic processes is needed for rigorous statistical approaches to the aforementioned problems. In this article we greatly extend the existing limit theory for persistence diagrams derived from sublevel set filtrations of discrete time stochastic processes. As a result, we understand the behavior of certain real-valued summaries of these random persistence diagrams---so-called persistence statistics---that are particularly relevant to machine learning and goodness-of-fit testing.

Work pertaining to the topology of sub/superlevel sets of random functions has its most prominent originator in \citet{rice1944}. 
Current work in the area of establishing results about the sublevel set ($0^{th}$) persistent homology of stochastic processes has focused on almost surely continuous processes, such as investigations into the expected persistence diagrams of Brownian motion \citep{chazal2018}; expected persistence diagrams of Brownian motion with drift \citep{baryshnikov2019}; and expectations for the number of barcodes and persistent Betti numbers $\beta^{s,t}_0$ of continuous semimartingales \citep{perez2022}. The formulas in \citet{perez2022}, save for the expected number of barcodes with lifetime greater than $\ell$, follow asymptotic formulas with $\ell$ tending to 0 or $\infty$. 

Though not overlapping entirely with our setting, some results for cubical persistent homology are applicable here. Notable relevant results include the strong law of large numbers (Theorem 5 in \citealp{hiraoka2022large}) for random cubical sets in $\R^d$ (where the quality of the strong law is vague convergence) and the central limit theorem for persistent Betti numbers of cubical filtrations in $\R^d$ found in Theorem 1.2.3 of \cite{miyanaga_diss}. These results imply a strong law of large numbers for the $(0^{th})$ sublevel set persistence diagrams of $m$-dependent stationary processes and a central limit theorem for the ($0^{th}$) persistent Betti numbers of sublevel sets of i.i.d. sequences, respectively. In this article, we establish a strong law of large numbers for functionals of persistence diagrams that is far more general than those existing in the setting of this article. We do so by normalizing the persistence diagrams so they become probability measures and by leveraging the tools of weak convergence. We also prove a central limit theorem for persistence diagrams evaluated on step functions using recent results for weakly dependent and potentially nonstationary triangular arrays, subject to standard dependence decay conditions on the underlying stationary sequence.

The quality of most strong laws of large numbers for persistence diagrams has been vague convergence, with \cite{hiraoka2018}, \cite{krebs2021}, and \cite{owada2022} tackling the geometric (i.e. {\v C}ech and Vietoris-Rips persistent homology) setting, and \cite{hiraoka2022large} addressing the cubical setting. Recently however, the authors of \cite{bobrowski2024} have employed the weak convergence ideas that we use here to prove a strong law of large numbers for the probability measure defined by death/birth ratios in a persistence diagram, for the geometric setting. In \cite{divol_polonik}---again in the geometric setting---the authors extend the set functions for which the strong law of \cite{hiraoka2018} holds to a class of unbounded functions. 

In Section~\ref{ss:unbounded}, we accomplish this extension as well in the setting of sublevel set persistent homology. We extend the strong law of large numbers (SLLN) for 0-dimensional sublevel set persistence diagrams of stationary processes \citep{hiraoka2022large} from continuous functions with compact support to a large class of unbounded functions. We achieve this based solely on minor conditions such as ergodicity and restrictions on the heaviness of the tails of the marginal distributions of our underlying stochastic process. We also remove the local dependence condition of \cite{hiraoka2022large}. In doing so, we answer an open question of \cite{chung2021} about the limiting empirical distribution of persistence diagram lifetimes for sublevel sets of discrete time stationary processes. For this specific setting, we also derive an explicit representation of the strong limit of our sublevel set persistent betti numbers in Proposition~\ref{p:exp_rep}, answering a query set forth in the conclusion to \cite{hiraoka_cube}. Finally, we extend the current state-of-the art result central limit theorem (CLT) for persistent Betti numbers of sublevel set filtrations of univariate discrete time stochastic processes (derived from Theorem 1.2.3 in \citealp{miyanaga_diss}) to finite-dimensional convergence and beyond the realm of i.i.d. observations.

This article proceeds in Section~\ref{s:background} with a treatment of persistent homology specialized to our setting, as well as details of our probabilistic setup. In Section~\ref{s:slln} the strong law of large numbers is stated and proved (Theorems~\ref{t:strong} and \ref{t:cont_conv}, on pages \pageref{t:strong} and \pageref{t:cont_conv}) and examples for which it holds are given for specific unbounded functionals of persistence diagrams in Corollary~\ref{c:pe_alps}. Beyond this, we derive some satisfying results in the case of i.i.d. stochastic processes in Corollary~\ref{e:cor_uniform} and state a Glivenko-Cantelli theorem for persistence lifetimes in Corollary~\ref{c:gliv_emp}. Finally in Section~\ref{s:clt}, we state the setting and results of our central limit theorem for persistence diagrams (Theorem~\ref{t:bst_clt}, on page \pageref{t:bst_clt}). We conclude with a brief discussion about the potential improvements and extensions of this work in Section~\ref{s:discuss}. Proofs of various lemmas and the central limit theorem are deferred to Section~\ref{s:clt_proof}.

\section{Background} \label{s:background}

We begin by discussing the necessary notions in topological data analysis---specifically 0-dimensional sublevel set persistent homology. From there, we detail crucial results for the representation of zero-dimensional sublevel set persistent homology for stochastic processes. 

Before continuing, let us make a brief note about notation. For a real numbers $x, y$ we define $x \wedge y := \min\{x,y\}$, $x \vee y := \max\{x,y\}$, and $(x)_+ := x \vee 0 = \max\{x, 0\}$. We set $\bar{\R} := [-\infty, \infty]$ and $\R_+ := [0, \infty)$. If $R$ is a set in some topological space, we denote $R^\circ$ the interior (i.e. largest open subset) of $R$ and $\partial R$ its boundary. We denote $B(z, \epsilon)$ to be the open Euclidean ball of radius $\epsilon > 0$ centered at $z$. If for a real sequence $(a_n)_{n \geq 1}$ and a positive sequence $(b_n)_{n \geq 1}$ we have $a_n/b_n \to 0$ as $n \to \infty$, we write $a_n = o(b_n)$; if there exists a $C > 0$ such that $|a_n| \leq Cb_n$ for $n$ large enough, we write $a_n = O(b_n)$.

\subsection{Homology} \label{ss:hom}

Recall that an (abstract) simplicial complex $K$ is a collection of subsets of a set $A$ with the property that it is closed under inclusion. Let $K$ be the graph (i.e. a special case of a simplicial complex) with vertex set $V = \{v_0, v_1, v_2, \dots\}$. To focus on the dependency structure of our stationary processes, we fix the edge set to be
\[
\{ \{v_0, v_1\}, \{v_1, v_2\}, \{v_2, v_3\}, \{v_3, v_4\}, \dots \}.
\] 
Fix a function $f: K \to (-\infty, \infty]$ that satisfies $\tau \subset \sigma$ $\Rightarrow$ $f(\tau) \leq f(\sigma)$ such that there exists some $n \geq 1$ such that $f(v_k) = \infty$ for $k > n$. We then define the \emph{subcomplex} $K(t) := \{\sigma \in K: f(\sigma) \leq t\}$ for $t \in \R$. It is clear that for $s \leq t$ we have $K(s) \subset K(t)$ and thus $\big(K(t)\big)_{t \in \R}$ defines a \emph{filtration} of graphs. For any $t \in \R$ we can assess the connectivity information of $K(t)$ by calculating its \emph{$0$-dimensional homology group $H_0(K(t))$}. We do so by initially forming two vector spaces $C_0$ and $C_1$ of all formal linear combinations of the vertices 
\[
C_0(K(t)) := \Bigg\{ \sum_{i: \, v_i \in K(t)} a_i v_i: a_i \in \mathbb{Z}_2 \Bigg\}
\]
and 
\[
C_1(K(t)) := \Bigg\{ \sum_{i: \, \{v_i, v_{i+1}\} \in K(t)} a_i \{v_i, v_{i+1}\}: a_i \in \mathbb{Z}_2 \Bigg \},
\]
where $\mathbb{Z}_2$ is the field of two elements $\{0,1\}$. The elements of $C_0(K(t))$ and $C_1(K(t))$ are called \emph{0-chains} and \emph{1-chains}, respectively. Addition of $i$-chains in $C_i(K(t))$ is done componentwise. To calculate $H_0(K(t))$ we need to specify the boundary map $\partial_1: C_1(K(t)) \to C_0(K(t))$, which is defined by 
\[
\partial_1\big(\{v_i, v_{i+1}\}\big) = v_i + v_{i+1}.
\]
We can extend this to an arbitrary $c \in C_1(K(t))$ by 
\[
\partial_1(c) = \sum_{i: \, \{v_i, v_{i+1}\} \in K(t)} a_i \partial_1(\{v_i, v_{i+1}\}).
\]

By analogy to the construction above, each vertex in $C_0(K(t))$ gets sent to 0 by $\partial_0$ so $Z_0(K(t)) := \ker \partial_0 = C_0(K(t))$. Defining $B_0(K(t)) := \partial_1\big(C_1(K(t))\big)$ (the image of $\partial_1$), we define the $0^{th}$ homology group as the quotient \emph{vector space},
\[
H_0(K(t)) := Z_0(K(t))/B_0(K(t)).
\]
A more general setup of homology with $\mathbb{Z}_2$ coefficients can be seen in Chapter 4 of \cite{edelsbrunner2010}. 

Note that the definition of $K$ and the boundary maps so defined establishes an exact correspondence with the homology of cubical sets in $\R$ with field coefficients in $\mathbb{Z}_2$ (cf. Chapter 2 of \citealp{kaczynski2006}). In particular, this is the reason we do not investigate higher-order simplices in $K$ as not only do the specific vertices and edges of $K$ preclude higher-dimensional cycles, but there are no elementary cubes of dimension higher than 1 in $\R$. Though we could present our treatment in the language of the cubical setting, we choose to use the equivalent simplicial formulation owing to its broader familiarity.

\subsection{Persistent homology and representations}

The vector spaces\footnote{Conventionally called groups, as coefficients may lie in $\mathbb{Z}$, for example.} $H_0(K(t))$ capture intuitive connectivity information---the elements of $H_0(K(t))$ are the equivalence classes of vertices that satisfy $v+v' \in B_0(K(t))$. More simply put, elements of $H_0(K(t))$ are vertices connected by a chain of edges. The information in $H_0(K(t))$ gives us useful information on the function $f$, but being able to assess how connected components (i.e. elements of $H_0(K(t))$) appear and merge as we vary $t$ would be better. We can do so by introducing the notion of \emph{persistent homology}.
\begin{figure}
\includegraphics[width=\textwidth]{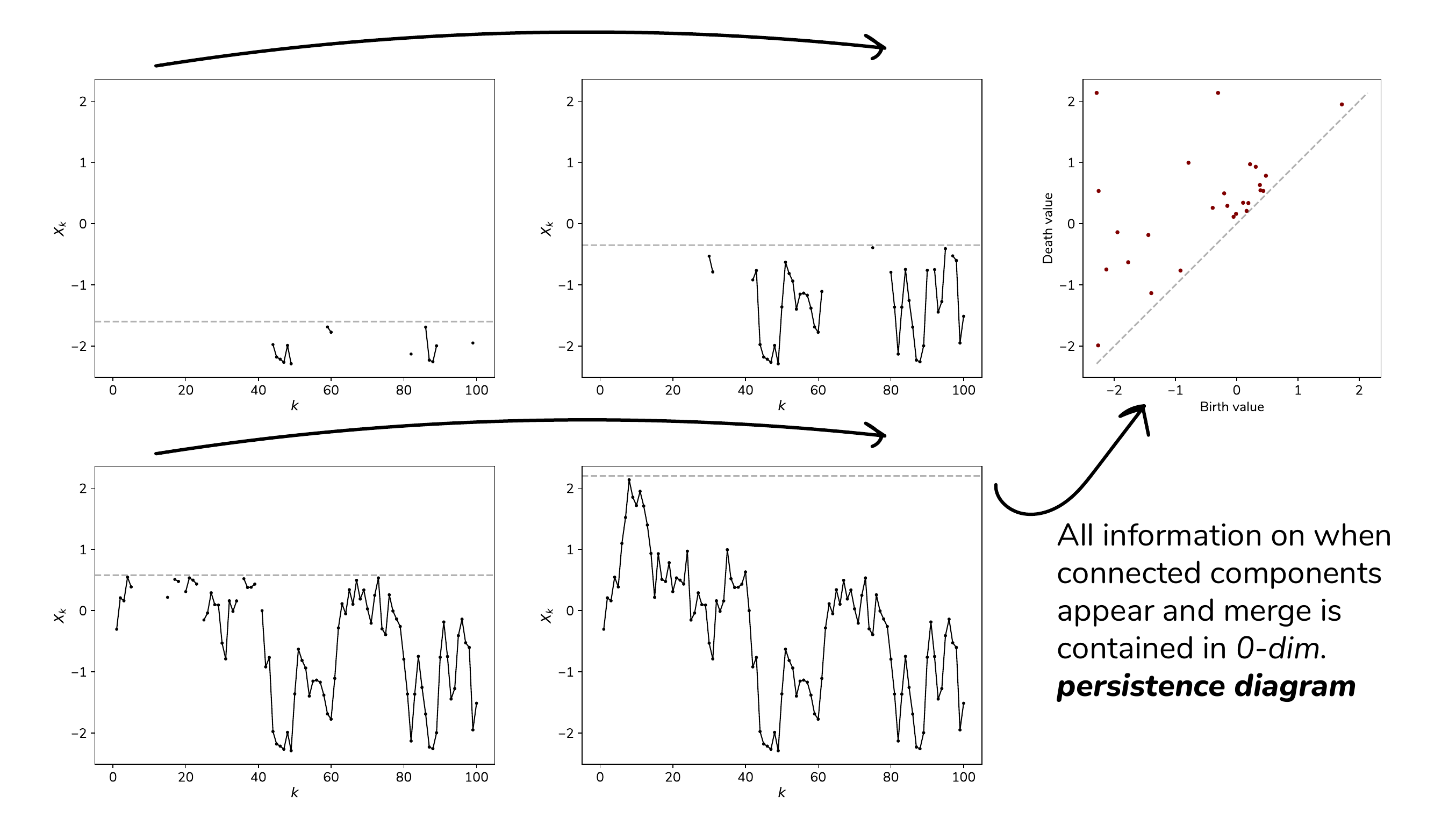}
\caption{Subcomplexes of the filtration $\big(K(t)\big)_{t \in \R}$ of a sample of 100 points from a $8$-dependent stationary Gaussian process along with its $0^{th}$ persistence diagram $PD_0$ (upper right).}
\label{f:stoch_pers}
\end{figure}
Given the inclusion maps $\iota_{s,t}: K(s) \to K(t)$, for $s \leq t$ there exist linear maps between all homology groups
\[
f^{s,t}_0: H_0(K(s)) \to H_0(K(t)), 
\]
which are induced by $\iota_{s,t}$. The \emph{persistent homology groups} of the filtration $(K(t))_{t \in \R}$ are the quotient vector spaces 
$$H^{s,t}_0(K) := \im\, f^{s,t}_0 \cong Z_0(K(s))/\big(B_0(K(t)) \cap Z_0(K(s))\big),$$
whose elements represent the cycles that are ``born'' in $K(s)$ or before and that ``die'' after $K(t)$. The dimensions of these vector spaces are the \emph{persistent Betti numbers} $\beta_0^{s,t}$. Heuristically, a connected component $\gamma \in H_0(K(s))$ is born at $K(s)$ if it appears for the first time in $H_0(K(s))$---formally, $\gamma \not \in H_0(K(r))$, for $r < s$. The component $\gamma \in H_0(K(s))$ dies entering $K(t)$ if it merges with an older class (born before $s$) entering $H_0(K(t))$. The $0^{th}$ persistent homology of $\X$, denoted $PH_0$, is the collection of homology groups $H_0(K(t))$, $t \in \R$, and maps $f^{s,t}_0$, for $-\infty < s \leq t \leq \infty$. $PH_0$ is generically called a \emph{persistence module}, which is a collection of vector spaces indexed by $\R$ with linear maps between each of them \citep{struct_stab}. Define the set $\Delta := \{(x,y) \in \bar{\R}^2: -\infty < x < y \leq \infty\}$. Because of how we have constructed the filtration $\big(K(t)\big)_{t \in \R}$, each $H_0(K(t))$ is finite-dimensional and there exists a unique multiset $PD_0$ consisting of points in $\Delta$ such that the following isomorphism\footnote{For more information on the precise nature of this direct sum and the theorem giving the existence of this decomposition, see Section 2.5 of \cite{struct_stab}.} holds:
$$
PH_0 \cong \bigoplus_{(b,d) \in PD_0} I(b, d).
$$
Here, each persistence module $I(b, d)$ is the collection of vector spaces $(U_t)_{t \geq \R}$ and maps $v^{s}_t$, $-\infty < s \leq t \leq \infty$ defined by 
$$
U_t := 
\begin{cases}
\mathbb{Z}_2 &\mbox{ if } b \leq t < d \\
0 &\mbox{ otherwise},
\end{cases}
$$
where $v^{s}_t$ the identity function on $\mathbb{Z}_2$ if $b \leq s \leq t < d$ and the zero map otherwise (\citealp{struct_stab}; see \citealp{hiraoka2018} for a less formal statement). As such, all of the information in the persistent homology groups is contained in the multiset $PD_0$, which is called the \emph{persistence diagram} \citep{edelsbrunner2010}. Explicitly, the $0^{th}$ persistence diagram $PD_0$ of the filtration $(K(t))_{t \in \R}$ consists of the points $(b, d)$ with multiplicity equal to the number of the connected components that are born at $K(b)$ and die entering $K(d)$. Often, the diagonal $y = x$ is added to this diagram, but we need not consider this here. 
In the same vein, we may represent $PD_0$ as a measure
\[
\xi_0 = \sum_{(b,d) \in PD_0} \delta_{(b,d)},
\] 
on $\Delta := \{(x,y) \in \bar{\R}^2: -\infty < x < y \leq \infty\}$. See Figure~\ref{f:stoch_pers} for an illustration of a persistence diagram associated to a sublevel set filtration of a given stochastic process.

\subsection{Probability and persistence}


Throughout the paper, let us fix a probability space $(\Omega, \mathcal{F}, \P)$. For random variables $X, X_1, X_2, \dots$ we write $X_n \Rightarrow X$ to convey that $X_n$ converges weakly to $X$, i.e. $\E[f(X_n)] \to \E[f(X)]$ for all bounded, continuous $f$. We write $X_n \overset{P}{\to} X$ to convey that $X_n$ converges in probability to $X$. 
We say an event $A \in \mathcal{F}$ occurs ``a.s.'' (almost surely), if $\P(A) = 1$. We use the term \emph{stationary} throughout this work to refer to the strict stationarity of invariance of finite-dimensional distributions under shifts. Namely, if $X_1, X_2, \dots$ are a sequence of random variables, then the sequence is stationary if for any $d, k \geq 1$ we have
$$
\P\big( (X_1, X_{2}, \dots, X_d) \in A\big) = \P\big( (X_{k+1}, X_{k+2}, \dots, X_{k+d}) \in A\big)
$$
for all Borel sets $A \subset \R^d$. Given a Borel set $A \subset \R^d$ and $k \geq 1$, we define the event 
$$
E := \{ \omega \in \Omega: (X_{k}(\omega), X_{k+1}(\omega), \dots, X_{k+d-1}(\omega)) \in A)
$$
and the shifted event $\tilde{E} := \{ \omega \in \Omega: (X_{k+1}(\omega), X_{k+2}(\omega), \dots, X_{k+d}(\omega)) \in A\}$. The event $E$ is said to be \emph{a.s. shift-invariant} if the symmetric difference of $E$ and $\tilde{E}$ satisfies
$\P(E \triangle \tilde{E}) = 0$ \citep{durrett2010}. A stationary sequence $X_1, X_2, \dots$ of random variables is said to be \emph{ergodic} if any a.s. shift-invariant event $E$ satisfies either $\P(E) = 0$ or $\P(E) = 1$.

Recall the definition of $K$ at the start of Section~\ref{ss:hom}. As we are interested in studying the stochastic behavior of persistence diagrams, we want to associate to each vertex $v_i \in K$ a random variable $X_i$ for each $i =0, 1, 2, \dots$. Consider a stationary sequence of random variables $X_1, X_2, \dots$ and define $X_0 \equiv \infty$. We then define the subcomplexes
\[
K_n(t) := \big\{\sigma \in K: \max_{v_i \in \sigma} X_{i,n} \leq t \big\}, \quad t \in \R
\]
where $X_{k, n} = \infty$ for $k > n$ and $X_{k,n} = X_k$ otherwise. Furthermore, set $K_n := \big(K_n(t)\big)_{t \in \R}$ and denote $\beta^{s,t}_{0,n}$, $- \infty < s \leq t \leq \infty$ to be the persistent Betti numbers of $K_n$. Critically, we have the following explicit representation of the $\beta^{s,t}_{0,n}$, whose proof is left to Section~\ref{s:clt_proof}.
\begin{proposition} \label{p:beta0st}
For the filtration $K_n$ we have that for $-\infty < s \leq t < \infty$ that 
\begin{equation}\label{e:beta0st}
\beta_{0,n}^{s,t} = \sum_{i = 1}^{n} \sum_{j=1}^{n-i+1} \mathbf{1} \bigg\{  \bigvee_{k=j}^{j+i-1} X_{k,n} \leq t, \bigwedge_{k=j}^{j+i-1} X_{k,n} \leq s \bigg \}\ind{X_{j-1, n} \wedge X_{j+i, n} > t}.
\end{equation}
and for $s \in \R$, $t = \infty$ we have 
$$
\beta_{0,n}^{s,t} = \mathbf{1} \bigg\{ \bigwedge_{k=1}^{n} X_{k,n} \leq s \bigg\}.
$$
\end{proposition}

Having brought forth the representation of persistent Betti numbers that will prove crucial to the results herein, we turn our attention to persistence diagrams. Let $\xi_{0,n}$ be the measure on $\Delta$ associated to the $0^{th}$ persistence diagram $PD_0$ of the filtration $K_n = \big(K_n(t)\big)_{t \in \R}$. Note that 
\[
\beta_{0,n}^{s,t} = \xi_{0,n}\big( (-\infty, s] \times (t, \infty] \big).
\]
If we let
\[
R = (s_1, s_2] \times (t_1, t_2],
\]
for $-\infty < s_1 \leq s_2 \leq t_1 \leq t_2 \leq \infty$, then 
\begin{equation} \label{e:flph}
\xi_{0,n}(R) = \beta_{0,n}^{s_2, t_1} - \beta_{0,n}^{s_2, t_2}  -  \beta_{0,n}^{s_1, t_1} + \beta_{0,n}^{s_1, t_2},
\end{equation}
due to the so-called ``Fundamental Lemma of Persistent Homology'' \citep{edelsbrunner2010}. If $R$ has the above representation, we will say that $s_1, s_2, t_1, t_2$ are the \emph{coordinates} of $R$. We define the class $\mathcal{R}$ of sets by 
\[
\mathcal{R} := \big\{ (s_1, s_2] \times (t_1, t_2]: -\infty < s_1 \leq s_2 \leq t_1 \leq t_2 \leq \infty \big\}.
\] 
An important result holds for the class $\mathcal{R}$, the proof of which is again deferred to Section~\ref{s:clt_proof}. 

\begin{lemma} \label{l:conv_det}
$\mathcal{R}$ is a convergence-determining class for weak convergence on $\Delta$ equipped with the Borel $\sigma$-algebra, $\mathcal{B}(\Delta)$. Namely, if $(\mu_n)_{n}$ and $\mu$ are probability measures on $\Delta$ and 
\[
\mu_n(R) \to \mu(R), \quad n \to \infty,
\]
for all $R \in \mathcal{R}$ such that $\mu(\partial R) = 0$, then 
\[
\mu_n \Rightarrow \mu, \quad n \to \infty. 
\]
Furthermore, for each probability measure $\mu$ on $\Delta$ there is a countable convergence-determining class $\mathcal{R}_\mu \subset \mathcal{R}$ for $\mu$.
\end{lemma}

To use the weak convergence arguments we employ below, we must also determine how the cardinality of $PD_0$, i.e. $\xi_{0,n}(\Delta)$, behaves. Namely, we show that the value $\xi_{0,n}(\Delta)$ is equal to the number of local minima of $X_{0,n}, X_{1,n}, \dots, X_{n,n}, X_{n+1,n}$. If the value $X_{i,n} = X_i < \infty$ is a (strict) local minimum (i.e. less than $X_{i-1,n}$ and $X_{i+1, n}$) it produces a new connected component in $H_0(K_n(X_i))$ that cannot be equivalent to anything else as $v_i$ belongs to no edges in $K_n$ until filtration time $X_{i-1,n} \wedge X_{i+1,n}$. One can verify this in the bottom right plot of Figure~\ref{f:stoch_pers} and the corresponding persistence diagram. We formalize this idea in the next proposition.
\begin{proposition} \label{p:diag_pts}
Suppose that $X_1, X_2, \dots$ is a stationary sequence of random variables with $\P(X_1 = X_2) = 0$. Then
$$
\xi_{0,n}(\Delta) = \sum_{i=1}^n \ind{X_{i,n} < X_{i-1,n} \wedge X_{i+1, n}}  
$$
\end{proposition}

\begin{proof}
The case when $n = 1$ is trivial, so suppose that $n \geq 2$. As the underlying stochastic process is stationary and $\P(X_1 = X_2) = 0$ then every value $X_1, X_2, \dots$ is distinct with probability 1. Let $a_i \equiv X_{(i), n}$ be the order statistics of $X_{1, n}, \dots, X_{n,n}$---which are distinct with probability 1---and let $v_{(i)}$ be the associated vertices (see above). If we define 
\[
K_i := K_n(a_i), \quad i = 1, \dots, n,
\]
with $K_0 = \emptyset$, then $K_0 \subset K_1 \subset \cdots \subset K_n$ and $K_{i+1}$ contains all the simplices of $K_i$ along with the 0-simplex $v_{(i+1)}$ and any edges containing it. Further define $a_0 \equiv -\infty$ and $a_{n+1} \equiv \infty$. If $0 \leq \ell < m \leq n+1$ then there are $\alpha$ points at $(a_\ell, a_m) \in \xi_{0,n}$ if and only if $\xi_{0,n}((a_{\ell-1}, a_\ell] \times (a_{m-1}, a_m]) = \alpha$---see p. 152 in \cite{edelsbrunner2010}. By Proposition~\ref{p:beta0st}, we have that 
\begin{align*}
&\xi_{0,n}((a_{\ell-1}, a_\ell] \times (a_{m-1}, a_m]) \\
&\ = \beta_{0,n}^{a_\ell, a_{m-1}} - \beta_{0,n}^{a_\ell, a_m}  -  \beta_{0,n}^{a_{\ell-1}, a_{m-1}} + \beta_{0,n}^{a_{\ell-1}, a_m} \\
&\ = \sum_{i = 1}^{n} \sum_{j=1}^{n-i+1} \mathbf{1} \bigg \{ \bigwedge_{k=j}^{j+i-1} X_{k,n} = a_\ell \bigg \} \\
&\ \times \Bigg[  \mathbf{1} \bigg\{  \bigvee_{k=j}^{j+i-1} X_{k,n} \leq a_{m-1}, X_{j-1, n} \wedge X_{j+i, n} > a_{m-1} \bigg\} \\
&\phantom{\ \times \Bigg[  \mathbf{1} \bigg\{  \bigvee_{k=j}^{j+i-1} X_{k,n} } - \mathbf{1} \bigg\{  \bigvee_{k=j}^{j+i-1} X_{k,n} \leq a_{m}, X_{j-1, n} \wedge X_{j+i, n} > a_{m} \bigg\} \Bigg].
\end{align*}
Now, $\xi_{0,n}(\Delta) = \sum_{\ell=1}^{n} \sum_{m=\ell+1}^{n+1} \xi_{0,n}((a_{\ell-1}, a_\ell] \times (a_{m-1}, a_m])$ so by cancelling sums---and the fact that $X_{j-1, n} \wedge X_{j+i, n} > a_{n+1} = \infty$ cannot happen---we have that
\begin{align}
&\sum_{i = 1}^{n} \sum_{j=1}^{n-i+1} \sum_{\ell=1}^{n} \mathbf{1} \bigg \{ \bigwedge_{k=j}^{j+i-1} X_{k,n} = a_\ell \bigg \} \notag \\
&\phantom{\sum_{i = 1}^{n} \sum_{j=1}^{n-i+1} \sum_{\ell=1}^{n-1}} \qquad \times \mathbf{1} \bigg\{  \bigvee_{k=j}^{j+i-1} X_{k,n} \leq a_\ell, X_{j-1, n} \wedge X_{j+i, n} > a_\ell \bigg\} \notag \\
&= \sum_{j=1}^n \sum_{\ell=1}^{n-1} \mathbf{1} \big \{ X_{j,n} = a_\ell, X_{j-1, n} \wedge X_{j+1, n} > a_\ell \}, \label{e:loc_min1}
\end{align}
because the only way the maximum and minimum of a collection of $i$ of random variables are identical is if they're constant---which is only possible if $i = 1$ as the $X_{i,n}, i =1, \dots, n$ are almost surely distinct. Furthermore, as $n \geq 2$ we must have $\ind{X_{j,n} = a_n, X_{j-1, n} \wedge X_{j+1, n} > a_n} = 0$. The desired formula follows from applying this same uniqueness to \eqref{e:loc_min1}.
\end{proof}

\begin{remark}
Note that all was necessary in proof above is that all the random variables took different values with probability 1. Therefore, if $X_{i,n} = x_{i,n}$ are deterministic real numbers with $x_{0,n} = x_{k,n} = \infty$ if $k > n$ such that $x_{1,n}, \dots, x_{n,n}$ are all distinct, then we have 
$$
\xi_{0,n}(\Delta) = \sum_{i=1}^n \ind{x_{i,n} < x_{i-1,n} \wedge x_{i+1, n}}.
$$
\end{remark}

To finish this section, we must introduce the restricted measure on the set $\tilde{\Delta} := \Delta \cap \R^2$---equipped with the usual Borel sub $\sigma$-algebra $\mathcal{B}(\tilde{\Delta})$---defined by 
\[
\tilde{\xi}_{0,n}(A) := \xi_{0,n}(A), \quad A \in \mathcal{B}(\tilde{\Delta}).
\]
Note that as $\Delta \cap \R^2$ is Borel subset of $\Delta$ that $\mathcal{B}(\tilde{\Delta}) \subset \mathcal{B}(\Delta)$. To reduce notational clutter, we will mostly write $\tilde{\xi}_{0,n}(\Delta)$ in place of $\tilde{\xi}_{0,n}(\tilde{\Delta})$ from here on out, unless otherwise noted. 

\section{Strong law of large numbers} \label{s:slln}

In this section we establish our strong law of large numbers for sublevel set persistence diagrams for a very broad class of sets and functions. We do this for the class of bounded, continuous functions initially via a weak convergence argument, and proceed to extend our result to a class of unbounded functions which are of great practical use in topological data analysis. Along the way, we give an explicit representation for the limiting persistent Betti numbers for i.i.d. sequences.

\begin{theorem}\label{t:strong}
Consider a stationary and ergodic sequence $\mathcal{X} = (X_1, X_2, \dots)$ where each $X_i$ has distribution $F$ and density $f$ such that $\P(X_1 = X_2) = 0$. For the random probability measure $\xi_{0, n}/\xi_{0,n}(\Delta)$ induced by $\X$ there exists a probability measure $\xi_0$ on $\Delta$ such that 
\[
\frac{\xi_{0, n}}{\xi_{0,n}(\Delta)} \Rightarrow \xi_0 \ \ \mathrm{a.s.},  \quad n \to \infty.
\] 
Additionally, if we define $\tilde{\xi}_0 \equiv \xi_0$ on $\mathcal{B}(\tilde{\Delta})$ then
\[
\frac{\tilde{\xi}_{0, n}}{\tilde{\xi}_{0,n}(\Delta)} \Rightarrow \tilde{\xi}_0 \ \ \mathrm{a.s.},  \quad n \to \infty.
\]

\end{theorem}
\begin{proof}
We will begin by establishing the almost sure convergence of the persistent Betti numbers $\beta_{0,n}^{s, t}/n$ for $-\infty < s \leq t \leq \infty$. Recall that 
\begin{align}
\frac{\beta_{0,n}^{s,t}}{n} &= \frac{1}{n} \sum_{i = 1}^{n}  \sum_{j=1}^{n-i+1} \mathbf{1} \bigg\{  \bigvee_{k=j}^{j+i-1} X_{k,n} \leq t, \bigwedge_{k=j}^{j+i-1} X_{k,n} \leq s \bigg \}\ind{X_{j-1, n} \wedge X_{j+i, n} > t} \notag \\
&= \frac{1}{n} \sum_{j=1}^n \sum_{i=1}^{n-j+1} \mathbf{1} \bigg\{  \bigvee_{k=j}^{j+i-1} X_{k,n} \leq t, \bigwedge_{k=j}^{j+i-1} X_{k,n} \leq s \bigg \}\ind{X_{j-1, n} \wedge X_{j+i, n} > t} \notag 
\end{align}
Define for $m \in \N \cup \{\infty\}$ the indicator random variables 
\begin{equation} \label{e:ymjst}
Y_{j,n}^m(s,t) := \sum_{i=1}^{m} \mathbf{1} \bigg\{  \bigvee_{k=j}^{j+i-1} X_{k,n} \leq t, \bigwedge_{k=j}^{j+i-1} X_{k,n} \leq s \bigg \}\ind{X_{j-1, n} \wedge X_{j+i, n} > t},
\end{equation}
and
$$
Y_{j}^m(s,t) := \sum_{i=1}^{m} \mathbf{1} \bigg\{  \bigvee_{k=j}^{j+i-1} X_{k} \leq t, \bigwedge_{k=j}^{j+i-1} X_{k} \leq s \bigg \}\ind{X_{j-1} \wedge X_{j+i} > t},
$$
which is simply $Y^m_{j,n}(s,t)$ without ``boundary effects''. We will take $X_0 \equiv 0$. If we fix $m$, we have for $n \geq m$ that
$$
\beta_{0,n}^{s,t} = \sum_{j=1}^n Y_{j,n}^{n-j+1}(s,t) \geq  \sum_{j=1}^{n-m+1}Y_{j,n}^{n-j+1}(s,t) \geq \sum_{j=1}^{n-m+1}Y_{j}^{m}(s,t),
$$
which yields 
$$
\beta_{0,n}^{s,t} \geq \sum_{j=2}^{n+1}Y_{j}^{m}(s,t) - (m+1).
$$
Similarly, we see that 
$$
\beta_{0,n}^{s,t} \leq 1 + \sum_{j=1}^n Y_{j,n}^{n-j}(s,t) \leq 2 + \sum_{j=2}^{n+1} Y_{j}^{\infty}(s,t),
$$
because 
$$
\sum_{j=1}^n \mathbf{1} \bigg\{  \bigvee_{k=j}^{n} X_{k,n} \leq t, \bigwedge_{k=j}^{n} X_{k,n} \leq s \bigg \}\ind{X_{j-1, n} > t} \in \{0, 1\}.
$$
It is readily observed for fixed $t \geq s$ that $Y_2^m(s,t), Y_3^m(s,t), \dots$ are indicator random variables and form a stationary and ergodic sequence, owing to Theorem 7.1.3 in \cite{durrett2010}, for example. Thus, Birkhoff's ergodic theorem implies that for any $m \in \N$ we have
\[
\E[Y_2^m(s,t)] \leq \liminf_{n \to \infty} \frac{\beta_{0,n}^{s,t}}{n} \leq \limsup_{n \to \infty} \frac{\beta_{0,n}^{s,t}}{n} \leq \E[Y^{\infty}_2(s,t)], \quad \mathrm{a.s.}
\]
The monotone convergence theorem then implies that 
\[
n^{-1}\beta_{0,n}^{s,t} \overset{\text{a.s.}}{\to} \E[Y_2^{\infty}(s,t)], \quad n \to \infty.
\]
To establish the convergence of $\xi_{0,n}(\Delta)/n$, it suffices to recall that from Proposition~\ref{p:diag_pts} the total number of points in the persistence diagram $\xi_{0,n}(\Delta)$ is equal to the number of local minima of $\X$. Therefore, the ergodic theorem once again implies that $\xi_{0,n}(\Delta)/n$ converges a.s. to $\P(X_2 < X_1 \wedge X_3)$ and 
\begin{equation*}
\frac{\xi_{0,n}\big((-\infty, s] \times (t, \infty]\big)}{\xi_{0,n}(\Delta)} \to \frac{\E[Y_2^{\infty}(s,t)]}{\P(X_2 < X_1 \wedge X_3)}, \ \ \mathrm{a.s.}, \quad n \to \infty.
\end{equation*}
(By our assumptions we must have that $P(X_2 < X_1 \wedge X_3) > 0$). Furthermore, for $R = (s_1, s_2] \times (t_1, t_2] \in \mathcal{R}$ we have by \eqref{e:flph} that
\begin{equation}\label{e:setR}
\frac{\xi_{0,n}\big(R\big)}{\xi_{0,n}(\Delta)} \to \frac{\E[Y_2^{\infty}(s_2,t_1) - Y_2^{\infty}(s_2,t_2) -Y_2^{\infty}(s_1,t_1) + Y_2^{\infty}(s_1,t_2)]}{\P(X_2 < X_1 \wedge X_3)}, \ \ \mathrm{a.s.}, \quad n \to \infty.
\end{equation}
From the limits we have just now described, define a set function $\bar{\xi}_0$ by 
\begin{equation} \label{e:setfunc}
\bar{\xi}_0\big((-\infty, s] \times (t, \infty])\big) := \frac{\E[Y_2^{\infty}(s,t)]}{\P(X_2 < X_1 \wedge X_3)}
\end{equation}
with the extension to $\mathcal{R}$ defined by the righthand side of \eqref{e:setR}. By Lemma~\ref{l:xi0}, there exists a unique probability measure $\xi_0$ on $(\Delta, \mathcal{B}(\Delta))$ that equals $\bar{\xi}_0$ on $\mathcal{R}$.
Finally, Lemma~\ref{l:conv_det} yields the existence of a countable convergence-determining class $\mathcal{R}_0 \subset \mathcal{R}$ for $\xi_0$. This implies that
\[
\P\Bigg( \lim_{n \to \infty} \frac{\xi_{0,n}(R)}{\xi_{0,n}(\Delta)} \to \xi_0(R), \text{ for any } R \in \mathcal{R}_0 \Bigg) = 1,
\]
so convergence for all sets in $\mathcal{B}(\Delta)$ with $\xi_0$-null boundary follows (with probability 1).

To finish the proof, note that it is the case\footnote{This fact implies that $\xi_0$ is supported on $\tilde{\Delta}$.} that $\tilde{\xi}_{0,n}(\tilde{\Delta}) \sim \xi_{0,n}(\Delta)$---as they both tend to infinity and differ by 1. Also, we have that for any set $A \in \mathcal{B}(\tilde{\Delta})$---which is also a Borel subset of $\Delta$---if $\xi_0(\partial A) = 0$, then almost surely
\[
\frac{\tilde{\xi}_{0,n}(A)}{\tilde{\xi}_{0,n}(\Delta)} \sim \frac{\xi_{0,n}(A)}{\xi_{0,n}(\Delta)} \to \xi_0(A), \quad n \to \infty.
\]
As $\xi_0(A) = \tilde{\xi_0}(A)$ for $A \in \mathcal{B}(\tilde{\Delta})$, the proof is finished. 
\end{proof}


\begin{remark}
In Theorem~\ref{t:strong} we assumed that $\P(X_1 = X_2) = 0$ in our stationary sequence, to ensure consecutive points are distinct, as stated in Proposition~\ref{p:diag_pts}. It seems straightforward to generalize this result to the situation where consecutive points can be identical, by accounting for this in the proof of Proposition~\ref{p:diag_pts}, and ensuring that the number of points in $\xi_{0,n}$ tends to infinity.
\end{remark}

Before seeing an example of the strong law in action, we will establish a result that will provide us an explicit representation of the limiting measure. Let us define the quantity
\begin{equation*}
p_i(s,t) := \P\Big( \bigcup_{k=1}^i \big\{ X_1 \leq t, \dots, X_k \leq s, \dots, X_i \leq t \} \Big),
\end{equation*}
which represents the probability that there is some index $k$ such that $X_k \leq s$ and all other random variables are less than or equal to $t$. In the setup with $X_i$ all i.i.d. with distribution function $F$ we have
\[
p_i(s,t) = F(t)^i - \big(F(t) - F(s)\big)^i
\]
and
\begin{align}
\E[\beta_{0,n}^{s,t}] &= \sum_{i = 1}^{n} \sum_{j=1}^{n-i+1} \P \bigg(  \bigvee_{k=j}^{j+i-1} X_{k,n} \leq t, \bigwedge_{k=j}^{j+i-1} X_{k,n} \leq s, \text{ and } X_{j-1, n} \wedge X_{j+i, n} > t \bigg) \notag \\
&= p_n(s,t) + 2p_{n-1}(s,t)(1-F(t)) \notag \\ \label{e:pb_rep}
&\phantom{= p_n(s,t) \ }+ \sum_{i = 1}^{n-2} \bigg(2p_i(s,t)(1-F(t)) + (n-i-1)p_i(s,t)(1-F(t))^2 \bigg).
\end{align}
We will assume that $0 < F(s) < 1$, as if $F(s) = 0$ then $\beta^{s,t}_{0,n} \equiv 0$ and if $F(s) = 1$ then $\beta^{s,t}_{0,n} \equiv 1$. Dividing \eqref{e:pb_rep} by $n$ we can see that 
\begin{align*}
\frac{\E[\beta_{0,n}^{s,t}]}{n} &\sim \frac{(1-F(t))^2}{n} \sum_{i = 1}^{n} (n-i+1)p_i(s,t) \\
&= \frac{(1-F(t))^2}{n} \sum_{i = 1}^{n} (n-i+1)[F(t)^i - \big(F(t)-F(s)\big)^i]
\end{align*}
as the other terms are finite or tend to zero upon dividing by $n$. Let us make the substitution $i = n-j+1$ and consider a general $a \in (0, 1]$ with $b = a^{-1}$. Thus,
\begin{align}
\sum_{i = 1}^{n} (n-i+1)a^i &= a^n \sum_{j = 1}^{n} jb^{j-1} \notag \\
&= a^n \Bigg[ \frac{nb^{n+1} - (n+1)b^n + 1}{(b-1)^2} \Bigg] \notag \\
&= \frac{nb - (n+1) + a^n}{(b-1)^2} \label{e:gsfin}
\end{align}
by differentiating $\sum_{i=1}^n x^i = (x^{n+1}-x)/(x-1)$ with respect to $x$. We have the following pleasing result for the limiting expectation for the persistent Betti number in this simplified i.i.d case.
\begin{proposition}\label{p:exp_rep}
For $X_i$ i.i.d. having distribution $F$, we have that 
\[
\frac{\E[\beta_{0,n}^{s,t}]}{n} \to \frac{(1-F(t)) F(s)}{1-F(t) + F(s)}, 
\]
for any $-\infty < s \leq t \leq \infty$ with $F(s) \in (0, 1)$ and $0$ otherwise.
\end{proposition}
\begin{proof}
Dividing by $n$ and taking the limit in \eqref{e:gsfin} for the two cases $a = F(t)$ and $a = F(t) -F(s)$ gives 
\[
\frac{(1-F(t))^2}{1/F(t)-1} = (1-F(t))F(t),
\]
and 
\[
\frac{(1-F(t))^2}{1/[F(t)-F(s)]-1} = \frac{(1-F(t))^2[F(t)-F(s)]}{1-F(t)+F(s)}.
\]
Simplifying the above two expressions yields the ultimate result.
\end{proof}

\begin{example}
If the stationary and ergodic sequence in Theorem~\ref{t:strong} is i.i.d, Proposition~\ref{p:exp_rep} shows we can characterize the limiting probability measure $\xi_0$ quite nicely. We note that
\begin{equation*} 
\xi_0\big( (-\infty, s] \times (t, \infty] \big) = \frac{3(1-F(t)) F(s)}{1-F(t) + F(s)}
\end{equation*}
for all $\infty < s \leq t \leq \infty$ as $\P(X_2 < X_1 \wedge X_3) = 1/3$. Therefore, $\xi_0$ admits a probability density
\begin{align*}
&- \frac{\partial^2}{\partial x \partial y} \Bigg[\frac{3(1-F(y))F(x)}{1-F(y)+F(x)}\Bigg] \notag \\
&\phantom{- \frac{\partial^2}{\partial x \partial y}}= \frac{6f(x)f(y)(1-F(y))F(x)}{(1-F(y)+F(x))^3}. 
\end{align*}
This density facilities the simulation of random variables according to the limiting persistence distribution $\xi^{\text{NULL}}_0$ in the case that $\mathcal{X}$ corresponds to i.i.d. noise. After a Monte Carlo random sample is generated from this distribution, we may test for ``significant'' points $(b,d)$ in the diagram $\xi_{0,n}$, based off of what we would expect from $\xi^{\text{NULL}}_0$. 

Of particular importance to us is the partial derivative 
\begin{align}
& \frac{\partial}{\partial x} \Bigg[\frac{3(1-F(y))F(x)}{1-F(y)+F(x)}\Bigg] \notag \\
&\phantom{- \frac{\partial^2}{\partial x \partial y}}= \frac{3f(x)(1-F(y))^2}{(1-F(y)+F(x))^2} \label{e:partial}.
\end{align}
If we set $y = x+\ell$, then \eqref{e:partial} evaluates to 
\begin{align*}
3\Bigg(\frac{1-F(x+\ell)}{1-F(x+\ell) + F(x)}\Bigg)^2f(x)
\end{align*}
Define $\Delta_\ell := \{(x,y) \in \Delta: y-x > \ell\}$ for $\ell \geq 0$. As a result of the above discussion, we have the following corollary.
\end{example}

\begin{corollary}\label{e:cor_uniform}
For $X_1, X_2, \dots$ i.i.d. with distribution function $F$ satisfying the conditions of Theorem~\ref{t:strong}, we have that 
\[
\xi_0(\Delta_\ell) = 3\E\Bigg[\bigg( \frac{1-F(X+\ell)}{1-F(X+\ell) + F(X)} \bigg)^2 \Bigg]
\]
where $X \overset{d}{=} X_1$.
\end{corollary}

\begin{example}
Corollary~\ref{e:cor_uniform} implies that for $F(t)$ uniform on $[0,1]$ we have for $0 < \ell < 1$ that
\begin{align*}
\xi_0(\Delta_\ell) &= 3\int_{0}^{1-\ell} \bigg(\frac{1-\ell - x}{1-\ell}\bigg)^2 \dif{x}. \\
&= 1-\ell,
\end{align*}
This is rather interesting, given that there is no \emph{a priori} reason that uniform noise should also produce asymptotically uniformly distributed persistence lifetimes.
\end{example}

Before addressing strong laws for unbounded functions, we conclude with a corollary of Theorem~\ref{t:strong}, establishing a Glivenko-Cantelli result for persistence lifetimes. We omit the proof of Corollary~\ref{c:gliv_emp} as it is proved in exactly the same manner as the Glivenko-Cantelli theorem---see Theorem~1.3 in \cite{dudley_uclt}.

\begin{corollary}\label{c:gliv_emp}
Suppose the conditions on the sequence $\X$ stated in Theorem~\ref{t:strong} hold. Then we have 
\[
\sup_{\ell \in [0, \infty)} \Bigg| \frac{\xi_{0, n}(\Delta_\ell)}{\xi_{0,n}(\Delta)} - \xi_0(\Delta_\ell)\Bigg| \to 0 \ \mathrm{a.s.}, \quad n \to \infty. 
\]
\end{corollary}
\vspace{11pt}

\subsection{SLLN for unbounded functions} \label{ss:unbounded}

At this point, we have established almost surely that 
\[
\tilde{\xi}_{0,n}(f)/\tilde{\xi}_{0,n}(\Delta) \to \tilde{\xi}_0(f),
\]
for any bounded, continuous real-valued function $f$ on $\tilde{\Delta}$, when $\tilde{\xi}_{0,n}$ is induced by a stationary and ergodic sequence of random variables (similar for $\xi_{0,n}$). In general, if $f$ is continuous, nonnegative function and $f \wedge M$ is the function that equals $M$ when $f \geq M$, then almost surely
\[
\tilde{\xi}_{0,n}(f \wedge M)/\tilde{\xi}_{0,n}(\Delta) = \frac{\sum_{(b,d) \in \tilde{\xi}_{0,n}} f(b,d) \wedge M}{\sum_{(b,d) \in \tilde{\xi}_{0,n}} 1} \to \int_{\tilde{\Delta}} f(x,y) \wedge M \, \tilde{\xi}_0(\dif{x}, \dif{y}), \quad n \to \infty,
\]
for all $M > 0$. Following this line of inquiry, we establish a result which yields convergence results for a large class of persistence statistics often seen in practice, including many of the functions for which convergence holds for geometric complexes in \cite{divol_polonik}, though we make no requirements on the behavior near the diagonal nor do we require polynomial growth. Prior to stating the result, it is necessary to define the notion of \emph{largely nondecreasing}. We say that an unbounded function $g: \R_+ \to \R_+$ is largely nondecreasing if there exists an $M > 0$ such that $\{x: g(x) \geq M\}$ is non-empty and $g$ is nondecreasing on $[g^{\leftarrow}(M), \infty)$ where $g^{\leftarrow}(M) = \inf \{x: g(x) \geq M\}$. Furthermore, recall that the function $g$ is \emph{coercive} if $g(x) \to \infty$ as $x \to \infty$.

\begin{theorem}\label{t:cont_conv}
Assume the conditions of Theorem~\ref{t:strong} and suppose that $f(b,d) = g(d-b)$ and $g: \R_+ \to \R_+$ is a continuous, coercive, and largely nondecreasing function with $\E\big[g(2|X_1|)^{1+\epsilon}\big] < \infty$ for some $\epsilon > 0$. If $\tilde{\xi}_0(f) < \infty$, then
\[
\tilde{\xi}_{0,n}(f)/\tilde{\xi}_{0,n}(\Delta) \to \tilde{\xi}_0(f), \ \ \mathrm{a.s.},  \quad n \to \infty.
\]
\end{theorem} 
\begin{proof}
Before beginning, fix any $M > 0$ such that $g$ is nondecreasing on $[g^{\leftarrow}(M), \infty)$. We will focus our proof on the case where a random variables with marginal distribution $F$ can take negative and positive values, but the proofs follow from a simplified version of the argument below when the support of $F$ is restricted to a half-line. To show that 
\[
\tilde{\xi}_{0,n}(f)/\tilde{\xi}_{0,n}(\Delta) \to \tilde{\xi}_0(f),
\]
for $f$ as in the statement of the theorem, it will suffice to first bound the quantity
\begin{align}
\frac{\tilde{\xi}_{0,n}(f)}{\tilde{\xi}_{0,n}(\Delta)} - \frac{\tilde{\xi}_{0,n}(f \wedge M)}{\tilde{\xi}_{0,n}(\Delta)} &= \frac{\tilde{\xi}_{0,n}\big( (f-M)_+\big)}{\tilde{\xi}_{0,n}(\Delta)} \notag \\
&= \tilde{\xi}_{0,n}(\Delta)^{-1}\sum_{\substack{(b,d) \in \xi_{0,n}, \\ f(b,d) \geq M}} f(b,d). \label{e:fbd}
\end{align}
Recall that $f(b,d) = g(d-b)$. In this situation, we have that the unnormalized form of \eqref{e:fbd} equals
\begin{align}
\sum_{d-b \geq g^{\leftarrow}(M)} g(d-b) &= \sum_{\substack{d-b \geq g^{\leftarrow}(M), \\ b \geq 0}} g(d-b) + \sum_{\substack{d-b \geq g^{\leftarrow}(M), \\ \ b < 0,\,  d < 0}} g(d-b) + \sum_{\substack{d-b \geq g^{\leftarrow}(M), \\ \ b < 0, \, d \geq 0}} g(d-b) \notag \\
&\leq \sum_{\substack{d \geq g^{\leftarrow}(M)}} g(d) + \sum_{\substack{-b \geq g^{\leftarrow}(M)}} g(-b) + \sum_{\substack{d-b \geq g^{\leftarrow}(M), \\ \ b < 0, \, d \geq 0}} g(2d) + g(-2b) \notag \\
&\leq \sum_{\substack{d \geq g^{\leftarrow}(M)}} g(d) + \sum_{\substack{-b \geq g^{\leftarrow}(M)}} g(-b) \notag \\
&\phantom{\leq \sum_{\substack{d \geq g^{\leftarrow}(M)}} g(d) \ } + \sum_{\substack{2\max\{d,-b\} \geq g^{\leftarrow}(M), \\ \ b < 0, \, d \geq 0}} g(2d) + g(-2b), \label{e:g2d2b}
\end{align}
because of the fact $g(d-b) \leq g\big(2\max\{d,-b\}\big) \leq g(2d) + g(-2b)$ when $b < 0$, $d \geq 0$ and we have $d-b \geq g^{\leftarrow}(M)$. Furthermore,
\begin{align*}
&\sum_{\substack{2\max\{d,-b\} \geq g^{\leftarrow}(M), \\ \ b < 0, \, d \geq 0}} g(2d) \\
&\qquad \qquad= \sum_{(b,d)\in\xi_{0,n}} g(2d)\ind{2\max\{d,-b\} \geq g^{\leftarrow}(M)}\big(\ind{d > -b} + \ind{d \leq -b}\big)\\
&\qquad \qquad= \sum_{(b,d)\in\xi_{0,n}} g(2d)\ind{2d \geq g^{\leftarrow}(M)}\ind{d > -b}  \\
&\qquad \qquad\qquad \qquad \qquad+ \sum_{(b,d)\in\xi_{0,n}} g(2d)\ind{-2b \geq g^{\leftarrow}(M)}\ind{d \leq -b} \\
&\qquad \qquad\leq \sum_{\substack{2d \geq g^{\leftarrow}(M)}} g(2d) + \sum_{\substack{-2b \geq g^{\leftarrow}(M)}} g(-2b).
\end{align*}
This occurs as $g(x) \leq g(y)$ if $y \geq g^{\leftarrow}(M) \vee x$. With a similar argument for the $g(-2b)$ term, we can see that \eqref{e:g2d2b} is bounded above by
\[
\sum_{\substack{2d \geq g^{\leftarrow}(M)}} 3g(2d) + \sum_{\substack{-2b \geq g^{\leftarrow}(M)}} 3g(-2b)
\]
By a similar argument to Proposition~\ref{p:diag_pts} occurs at $d=X_i$ if and only $X_i$ is a local maxima. Birkhoff's ergodic theorem then implies that 
\[
\sum_{2d \geq g^{\leftarrow}(M)} g(2d)/n \to \E\big[ g(2X_2)\ind{X_2 > X_1 \vee X_3}\ind{2X_2 \geq g^{\leftarrow}(M)} \big], \quad \text{a.s.}, 
\]
as $n \to \infty$. H{\"o}lder's inequality then implies that for $p > 1$ and $q = p/(p-1)$,
\begin{align*}
&\E\big[ g(2X_2)\ind{X_2 > X_1 \vee X_3}\ind{2X_2 \geq g^{\leftarrow}(M)} \big] \\
&\phantom{\E\big[ g(2X_2)\ind{X_2 > X_1 \vee X_3}}\leq \Bigg( \E[g(2|X_2|)^p] \Bigg)^{1/p} \Bigg( \P(2|X_2| \geq g^{\leftarrow}(M)) \Bigg)^{1/q} \label{e:h_ineq_d}
\end{align*}
By assumption, $\E[g(2|X_2|)^p] < \infty$ for some $p > 1$, so that coercivity of $g$ entails we may choose $M > 0$ large enough such that 
\[
\E\big[ g(2X_2)\ind{X_2 > X_1 \vee X_3}\ind{2X_2 \geq g^{\leftarrow}(M)} \big] < \epsilon\P(X_2 < X_1 \wedge X_3)/18.
\]
Therefore, for such an $M$ we have 
\[
\limsup_{n \to \infty} \sum_{2d \geq g^{\leftarrow}(M)} 3g(2d)/\tilde{\xi}_{0,n}(\Delta) < \epsilon/6, \ \ \text{a.s.}
\]
A similar argument holds for the term
\[
\sum_{\substack{-2b \geq g^{\leftarrow}(M)}} 3g(-2b),
\]
so the additivity of $\limsup$ furnishes that 
\[
\limsup_{n \to \infty} \sum_{d-b \geq g^{\leftarrow}(M)} g(d-b)/\tilde{\xi}_{0,n}(\Delta) < \epsilon/3, \ \ \text{a.s.}
\]
By Theorem~\ref{t:strong} and the triangle inequality, it remains to show that 
\[
\tilde{\xi}_{0}\big( (f-M')_+\big) < \epsilon/3
\]
for some $M' \geq M$, which follows from $\tilde{\xi}_0(f) < \infty$. 
\end{proof}

If all $X_i$ are nonnegative, we have an easy corollary to Theorem~\ref{t:cont_conv}. We omit the proof as it follows directly from the one above. 

\begin{corollary}
If $X_i \geq 0$ for all $i = 1, 2, \dots$ then Theorem~\ref{t:cont_conv} holds for $f(b,d) = g(d+b)$.
\end{corollary}

The utility of Theorem~\ref{t:cont_conv} can be seen in the following section.

\subsection{Strong law of large numbers: two examples}
Strong laws of large numbers can be established from Theorem~\ref{t:cont_conv} for various quantities used in topological data science called \emph{persistence statistics}. For instance, we have a strong law of large numbers for degree-$p$ total persistence\footnote{See \cite{cohen2010} for a definition and \cite{divol_polonik} for the geometric complex result}, provided that 
\[
\E\big[|X_1|^{p+\epsilon}\big] < \infty.
\]
A more difficult example is \emph{persistent entropy} \citep{merelli2015, atienza2020}. Persistent entropy has been used as part of a suite of statistics in the studies of \cite{chung2021, chung_eat, chung2024} and \cite{bclr}, as well as to detect activation in the immune system \citep{rucco2016}, and to detect structure in nanoparticle images \citep{detectda, crozier2024}. The definition of persistent entropy (excluding the longest barcode) is
\[
E(X_1, \dots, X_n) \equiv E_n := -\sum_{(b,d) \in \tilde{\xi}_{0,n}} \frac{d-b}{L_n}\log \Bigg( \frac{d-b}{L_n} \Bigg),
\]
where $L_n := \sum_{(b,d) \in \tilde{\xi}_{0,n}} d-b$. Alternatively, we may represent $E_n$ as 
\begin{align*}
-L_n^{-1} \sum_{(b,d) \in \tilde{\xi}_{0,n}} (d-b)\log (d-b) + \log L_n,
\end{align*}
which will be of use in the corollary below.

Another nontrivial statistic of interest is the ALPS statistic, defined in \cite{detectda} and utilized in \cite{detectda}, \cite{crozier2024}, and \cite{bclr}. Its representation is
\[
A(X_1, \dots, X_n) \equiv A_n := \int_0^{\infty} \log \xi_{0,n}(\Delta_\ell) \, \dif{\ell},
\]
and we define a truncation of the ALPS statistic as $A_n^L := \int_0^{L} \log \xi_{0,n}(\Delta_\ell) \, \dif{\ell}.$ Before continuing, let us define $f_e(b,d) = (d-b)\log (d-b)$ and $f_I(b,d) = d-b$. Both $f_e + 1$ and $f_I$ are continuous, coercive, and largely nondecreasing in $d-b$.

\begin{corollary}\label{c:pe_alps}
Assuming the conditions of Theorems~\ref{t:strong} and \ref{t:cont_conv}, we have that 
\[
E_n - \log \tilde{\xi}_{0,n}(\Delta) \to \frac{\tilde{\xi}_0(f_e)}{\tilde{\xi}_0(f_I)} + \log \tilde{\xi}_0(f_I), \ \ \mathrm{a.s.},
\]
and for any $L > 0$ with $\xi_{0}(\Delta_L) > 0$ we have 
\[
L\log \xi_{0,n}(\Delta) - A_n^L \to -\int^L_0 \log \xi_0(\Delta_\ell) \dif{\ell}, \ \ \mathrm{a.s.},
\]
as $n \to \infty$. That is, the sublevel set persistent entropy and the ALPS statistic of a stationary and ergodic process converge almost surely. 
\end{corollary}

\begin{proof}
The proof follows fairly simply from Theorem~\ref{t:cont_conv}. We know that
\begin{equation*} 
E_n = \frac{-\tilde{\xi}_{0,n}(f_e+1) + \tilde{\xi}_{0,n}(\Delta)}{\tilde{\xi}_{0,n}(f_I)} + \log \tilde{\xi}_{0,n}(f_I).
\end{equation*}
Subtracting $\log \tilde{\xi}_{0,n}(\Delta)$ and applying Theorem~\ref{t:cont_conv} yields a limit of 
\[
\frac{-\tilde{\xi}_{0}(f_e+1) + 1}{\tilde{\xi}_{0}(f_I)} + \log \tilde{\xi}_{0}(f_I),
\]
which finishes the proof, as $\tilde{\xi}_{0}$ is a probability measure. For the ALPS statistic, we see that 
\[
L\log \xi_{0,n}(\Delta) - A_n^L = \int^L_0 \log \bigg( \frac{\xi_{0,n}(\Delta)}{\xi_{0,n}(\Delta_\ell)} \bigg) \dif{\ell}.
\]
If we fix a positive $\epsilon < \xi_{0}(\Delta_L)$, Corollary~\ref{c:gliv_emp} implies that for $n \geq N(\omega)$ ($N$ depending on the sample point $\omega \in \Omega$), we have 
\[
\log \bigg( \frac{\xi_{0,n}(\Delta)}{\xi_{0,n}(\Delta_\ell)}\bigg) \leq -\log (\xi_0(\Delta_\ell) - \epsilon) \leq -\log (\xi_0(\Delta_L) - \epsilon), 
\]
for all $\ell \in [0, L]$. Therefore, the bounded convergence assumption holds for all $\omega \in \Omega$ such that convergence holds. Hence, our result follows almost surely. \end{proof}

Having demonstrated our strong law of large numbers for persistence diagrams, and its ramifications, we now turn our attention to the central limit theorem.

\section{Central limit theorem} \label{s:clt}

In this section, we prove a central limit theorem for the integral $\xi_{0,n}(f)$, where $f$ is a step function. This follows from proving a CLT for linear combinations of persistent Betti numbers $\beta^{s,t}_{0,n}$ using the Lindeberg method for weakly dependent triangular arrays given in \cite{neumann2013}. The desired result will follow as a consequence of demonstrating
\begin{equation*}
n^{-1/2}\sum_{l=1}^m a_l \Big( \beta_{0,n}^{s_l, t_l} - \E[\beta_{0,n}^{s_l, t_l} ]\Big).
\end{equation*}
obeys a central limit theorem when $\X_n$ obeys weak dependence conditions (to be specified below) and $a_1, \dots, a_m$ are arbitrary real numbers. The reason for this is that if $R_l = (s_1, s_2] \times (t_1, t_2]$ then 
\[
\mathbf{1}_{R_l} = \mathbf{1}_{(-\infty, s_2]\times(t_1, \infty]} - \mathbf{1}_{(-\infty, s_2]\times(t_2, \infty]} - \mathbf{1}_{(-\infty, s_1]\times(t_1, \infty]} + \mathbf{1}_{(-\infty, s_1]\times(t_2, \infty]}.
\]
The Cr{\'a}mer-Wold device also provides us with finite-dimensional weak convergence as an added benefit.

As for the aforementioned notions of weak dependence, the one we employ is that of $\rho$-mixing. To begin, note that for any two sub-$\sigma$ algebras $\mathcal{A}, \mathcal{B} \subset \mathcal{F}$ we define 
\[
\rho(\mathcal{A}, \mathcal{B}) := \sup_{\substack{X \in L^2(\mathcal{A}), \,Y \in L^2(\mathcal{B})}}  \big|  \text{Corr}(X, Y) \big|,
\]
where $L^2(\mathcal{A})$ (resp. $L^2(\mathcal{B})$) is the space of square-integrable $\mathcal{A}$-measurable (resp. $\mathcal{B}$-measurable) random variables\footnote{For random variables $X, Y$ the value $\text{Corr}(X,Y) = \Cov(X,Y)/\sqrt{\Var(X)\Var(Y)}$.}.  Furthermore, we define 
\[
\rho_\X(k) := \sup_{m \in \N} \rho\big(\sigma(X_1, \dots, X_m), \sigma(X_{m+k}, X_{m+k+1}, \dots) \big),
\]
so that the stochastic process $\X = (X_1, X_2, \dots)$ is said to be $\rho$-mixing if $\rho_\X(k) \to 0$ as $k \to \infty$. For our limit theorems, we will require that $\sum_{k=1}^{\infty} \rho_\X(k) < \infty$, which implies $\rho$-mixing. More details on $\rho$-mixing and other mixing conditions can be seen in \cite{bradley2005}. Another particularly important condition for our proofs is that our stationary process obeys a certain condition on the probability distributions of the partial maxima decaying sufficiently quickly. This serves to limit any percolation-esque phenomena that would preclude a central limit theorem.

\begin{definition}
A stationary stochastic process $\X = (X_1, X_2, \dots)$ with marginal distribution function $F$ is said to be \emph{max-root summable} if for all $t$ with $F(t) < 1$ we have
\begin{equation*} 
\sum_{i=1}^\infty i \sqrt{\P(X_1 \leq t, \dots, X_i \leq t)} < \infty. 
\end{equation*}
\end{definition}
Before stating our main theorem, we will establish conditions on the stochastic process that guarantee max-root summability. 

\begin{proposition} \label{p:max_root_suff}
Suppose that $\X$ is a stationary stochastic process. If there is some $\epsilon > 0$ s.t.
$$
\P(X_1 \leq t, \dots, X_n \leq t) = O(n^{-4-\epsilon}),
$$
for all $t$ with $F(t) < 1$, then $\X$ is max-root summable.
\end{proposition}

\begin{proof}
If the condition above holds there is some $C_t$ such that 
\[
n\sqrt{\P(X_1 \leq t, \dots, X_n \leq t)} \leq \sqrt{C_t}n^{-1-\epsilon/2}, 
\]
the right-hand side of which is clearly summable.
\end{proof}

\begin{example}\label{e:markov}
Suppose that $\X$ is a (stationary) Markov chain with transition kernel $P$ such that for every $t$ with $F(t) < 1$ there is some $\eta_t > 0$ that satisifies
$$
\sup_{x \leq t} P\big(x, (-\infty, t]\big) \leq 1-\eta_t.
$$
By Theorem 3.4.1 in \cite{meyn2009}, we have that 
\begin{align*}
&\P(X_1 \leq t, \dots, X_n \leq t) \\
&\qquad\qquad= \int_{x_1 \leq t} \cdots \int_{x_{n-1} \leq t} F(\dif x_1)P(x_1, \dif x_2) \cdots P(x_{n-2}, \dif x_{n-1}) P\big(x_{n-1}, (-\infty, t]\big) \\
&\qquad\qquad\leq \int_{x_1 \leq t} \cdots \int_{x_{n-1} \leq t} F(\dif x_1)P(x_1, \dif x_2) \cdots P(x_{n-2}, \dif x_{n-1})(1-\eta_t).
\end{align*}
Therefore, induction furnishes that 
\[
\P(X_1 \leq t, \dots, X_n \leq t) \leq F(t)(1-\eta_t)^{n-1},
\]
and the condition in Proposition~\ref{p:max_root_suff} can be simply established. 
\end{example}

\begin{example}\label{e:mdep}
Suppose that $\X$ is stationary and $m$-dependent, i.e. $\psi_{\X}(k) = 0$ for all $k \geq m+1$. Then we have 
\begin{align*}
\P(X_1 \leq t, \dots, X_n \leq t) &\leq \P(X_1 \leq t, X_{m+2} \leq t, \dots, X_{\lfloor \frac{n-1}{m+1} \rfloor (m+1)+1} \leq t) \\
&= F(t)^{\lfloor \frac{n-1}{m+1} \rfloor + 1}.
\end{align*}
Because $F(t) = 0$ establishes max-root summability trivially, we take $0 < F(t) < 1$. Then as $(\lfloor \frac{n-1}{m+1} \rfloor + 1) \log \big[1/F(t)\big] \geq k \log n$ for any $k > 0$ and $n$ large enough, then the condition in Proposition~\ref{p:max_root_suff} is established.
\end{example}

To establish our CLT (Theorem~\ref{t:bst_clt} below), we first need to assess the limiting behavior of the covariance, the proof of which is deferred to the end of Section~\ref{s:clt_proof}.

\begin{proposition}\label{p:cov}
Let $\X$ be a stationary stochastic process that is max-root summable and satisfies the condition $\sum_{k=1}^{\infty} \rho_\X(k) < \infty$. Assume further that the marginal distribution of $X_i$ is continuous with distribution $F$. Suppose that $-\infty < s_i \leq t_i \leq \infty$ for $i=1,2$ with $F(s_1 \wedge s_2) > 0$ and $F(t_1 \vee t_2) < 1$.
\begin{align*}
&\lim_{n \to \infty} n^{-1}\mathrm{Cov}\Big(\beta_{0,n}^{s_1, t_1}, \beta_{0,n}^{s_2, t_2}\Big) \\
&\qquad \qquad = \mathrm{Cov}\big(Y_2^{\infty}(s_1, t_1), Y_2^{\infty}(s_2, t_2) \big) \\
&\qquad \qquad + \sum_{k=1}^{\infty} \bigg[ \mathrm{Cov}\big(Y_2^{\infty}(s_1, t_1), Y_{2+k}^{\infty}(s_2, t_2) \big) + \mathrm{Cov}\big(Y_{2+k}^{\infty}(s_1, t_1), Y_{2}^{\infty}(s_2, t_2) \big) \bigg].
\end{align*}
where the terms $Y_j^{\infty}(s,t)$ are defined at  \eqref{e:ymjst} respectively.
\end{proposition}

With this all at hand, we may finally state the central limit theorem.

\begin{theorem}\label{t:bst_clt}
Let $\X$ be a stationary stochastic process that is max-root summable and satisfies the condition $\sum_{k=1}^{\infty} \rho_\X(k) < \infty$. Assume further that the marginal distribution of $X_i$ is continuous with distribution $F$. Then for any function $f = \sum_{l=1}^m a_l \mathbf{1}_{R_l}$ with $a_l \in \R$ and $R_l \in \mathcal{R}$, $l = 1,\dots, m$, if the corners $(s,t)$ of the rectangles satisfy $F(s) > 0$ and $F(t) < 1$ we have:
\[
n^{-1/2}\big(\xi_{0,n}(f) -  \E[\xi_{0,n}(f)]\big) \Rightarrow N(0, I_f), 
\]
and if each of the coordinates of $R_l$ lie in $\R$ for $l = 1,\dots,m$ then
\[
n^{-1/2}\big(\tilde{\xi}_{0,n}(f) -  \E[\tilde{\xi}_{0,n}(f)]\big) \Rightarrow N(0, I_f), 
\]
as $n \to \infty$, where $I_f$ is a nonnegative constant depending on $f$.
\end{theorem}


We defer the proof to Section~\ref{s:clt_proof}.
%
\section{Discussion}\label{s:discuss}

In this paper, we have demonstrated a strong law of large numbers for a large class of integrals with the respect to the random measure induced by the $0^{th}$ sublevel set persistent homology of general stationary and ergodic processes. We also proved a central limit theorem for the same random measure for a large class of step functions. As the SLLNs---by consideration of the negated process $-X_1, -X_2, \dots$---also pertain to superlevel sets, it would be interesting to consider the limiting behavior of the persistent homology of the extremes of a stationary stochastic process; the reason is due to the natural connection between the superlevel set value $\beta_{0,n}^{u_n(\tau), u_n(\tau)}$ (number of connected components above levels $u_n(\tau)$, $\tau \geq 0$) and the clusters of exceedances seen in the extreme value theory literature (see chapter 6 of \citealp{kulik2020}).

Two potential improvements for this paper seem to lie in the weakening of conditions and the augmentation of the class of functions for which the central limit theorem holds (Theorem~\ref{t:bst_clt}). There are likely only improvements to be made in the latter case, as the $\sum_{k = 1}^{\infty} \rho_{\X}(k) < \infty$ condition is only slightly stronger than the slowest mixing rate of $\sum_{k = 1}^{\infty} k^{-1}\rho_{\X}(k) < \infty$ for a conventional CLT to hold for a stationary sequence \citep{bradley1987}. The improvement of the second objective seemingly depends on a more precise treatment of the covariance in Proposition~\ref{p:cov}, which is rather tedious as it stands. Nonetheless, such improvements would see utility as the class of functions of persistence diagrams used in practice are large, which is what motivated Section~\ref{ss:unbounded} (and this paper) to begin with. Expanding the CLT results to a functional CLT for the persistent Betti numbers (as in \citealp{krebs_fclt}) may yield some progress towards this end, but we leave all the pursuits mentioned in these last two paragraphs for future work. 

\section{Proofs} \label{s:clt_proof}
In this section, we give the proof of Proposition~\ref{p:beta0st}, Lemma~\ref{l:conv_det}, state and prove Lemma~\ref{l:xi0}, and provide the proof of the central limit theorem, Theorem~\ref{t:bst_clt} (including Proposition~\ref{p:cov}). We begin with the proof of Proposition~\ref{p:beta0st}, which underpins our results. 
\begin{proof}[Proof of Proposition~\ref{p:beta0st}]

Fix both $-\infty < s < t < \infty$ and $n \geq 1$ and define the quantity $L$ to be the number of instances where $X_{k,n} > t$, $k = 1, \dots, n+1$. Now set $E^{(0)}_t := 0$ and define
$$
E^{(\ell)}_t = \min \bigg\{k > E^{(\ell-1)}_t: X_{k, n} > t\bigg\}, \quad \ell = 1, \dots, L.
$$
Note that by this definition, $E^{(L)}_t = n+1$. Define 
$$
A_\ell := \{E^{(\ell-1)}_t+1 \leq k \leq E^{(\ell)}_t-1: X_{k,n} \leq s\}
$$
for $\ell = 1, \dots, L$ and let $h_1, \dots, h_J$ be the subsequence of $1, \dots, L$ such that $A_{h_j} \neq \emptyset$ for $j = 1, \dots,  J$. Finally, define $\ell_j := \inf{A_{h_j}}$. This yields the collection of vertices (chains) $v_{\ell_1}, \dots, v_{\ell_J}$ in $Z_0(K_n(s)) \equiv C_0(K_n(s))$ 

Now, let us consider a generic chain 
$$
c = \sum_{\ell=1}^n a_\ell v_\ell \ind{X_{\ell, n} \leq s}
$$
in $Z_0(K_n(s))$. By the definition of quotient vector space, two classes 
$$
[c], [c'] \in Z_0(K_n(s))/\big(B_0(K_n(t)) \cap Z_0(K_n(s))\big)
$$
of chains $c, c' \in Z_0(K_n(s))$ are equivalent if $c+c' \in B_0(K_n(t))$. That is, $[c] = [c']$ if 
\begin{align*}
c + c' &= \sum_{\ell=1}^{n-1} b_{\ell} \ind{X_{\ell,n} \vee X_{\ell+1,n} \leq t} (v_{\ell} + v_{\ell+1})
\end{align*}
for some $b_1, \dots, b_{n-1} \in \mathbb{Z}_2$. 
We assert that the $[v_{\ell_j}]$, $\ell = 1, \dots, J$ are linearly independent. Consider that 
$$
\sum_{j=1}^{J} a_j [v_{\ell_j}] = \Bigg[\sum_{j=1}^{J} a_jv_{\ell_j} \Bigg] = [0], 
$$
if and only if 
\begin{equation}\label{e:linind}
\sum_{j=1}^{J} a_jv_{\ell_j} = \sum_{\ell=1}^{n-1} b_{\ell} \ind{X_{\ell,n} \vee X_{\ell+1,n} \leq t} (v_{\ell} + v_{\ell+1}),
\end{equation}
for some $b_\ell \in \mathbb{Z}_2$, $\ell = 1, \dots, n-1$. (Note that we omit the indicator terms $\ind{X_{\ell_j,n} \leq s}$ to reduce notational clutter). Suppose that there is some $a_j = 1$ in \eqref{e:linind}. To include $v_{\ell_j}$ in the lefthand side of \eqref{e:linind}, we must have either $b_{\ell_j-1} = 1$ or $b_{\ell_j+1} = 1$ on the righthand side (but not both). Therefore, either $v_{\ell_j-1}$ or $v_{\ell_j+1}$ is included on the righthand side. These terms cannot be on the left-hand side of \eqref{e:linind} as it holds that $\ell_{j+1} \geq \ell_j + 2$ for all $j = 1, \dots, J-1$. \\

If $X_{\ell_j-1, n} > t$ (resp. $X_{\ell_j+1, n} > t$) then we cannot get rid of this vertex, as we cannot include $v_{\ell_j-2} + v_{\ell_j-1}$ (resp. $v_{\ell_j+1} + v_{\ell_j+2}$) in the righthand side of \eqref{e:linind} (as this would require $X_{\ell_j - 1, n} \leq t$ or $X_{\ell_j+1, n} \leq t$). Thus we must have $a_j = 0$. If $X_{\ell_j-1, n} \leq t$ (resp. $X_{\ell_j+1, n} \leq t$), to get rid of $v_{\ell_j-1}$ or (resp. $v_{\ell_j+1}$) we must have $b_{\ell_j-2} = 1$ (resp. $b_{\ell_j+2} = 1$) and thus include $v_{\ell_j-2}$ or $v_{\ell_j+2}$. Continuing this way, we will end up with a $v_{k}$ on the righthand side of \eqref{e:linind} such $\ell_{j-1} < k < \ell_{j+1}$ and $X_{k,n} > t$. To get rid of such a vertex, we would have to have $X_{k,n} \leq t$, a contradiction. \\

Therefore we can conclude that $a_j = 0$ for all $j =  1, \dots, J$ and that the $[v_{\ell_j}]$ are linearly independent. It only remains to show that they span the quotient space $Z_0(K_n(s))/\big(B_0(K_n(t)) \cap Z_0(K_n(s))\big)$. Considering $[c]$ as above (once again omitting the indicator functions), we must have that 
$$
\sum_{\ell=1}^{n-1} a_{\ell}[v_\ell] = \sum_{j=1}^{J} a'_j[v_{\ell_j}]
$$
for some $a'_j \in \mathbb{Z}_2$, $j = 1, \dots, J$. It will suffice to show that $[v_{i_1} + \cdots + v_{i_M}] = a'_j[v_{\ell_j}]$ for any $i_1, \dots, i_M \in A_{h_j}$. It is straightforward to see that $[v_{\ell_j}] = [v_{i}]$ for any other $i \in A_{h_j}$ as 
$$
v_{\ell_j} + v_i = \sum_{k=\ell_j}^{i-1} \ind{X_{k,n} \vee X_{k+1, n} \leq t}(v_k + v_{k+1})
$$
where $\ind{X_{k,n} \vee X_{k+1, n} \leq t} = 1$ for all $k = \ell_j, \dots, i-1$ because 
$$
E^{(h_j-1)}_t+1 \leq \ell_j \leq i \leq E^{(h_j)}_t-1.
$$ 
If we suppose that $M = 2N$ for some positive integer $N$, then 
$$
[v_{i_1} + \cdots + v_{i_M}] = \sum_{m=1}^N ([v_{i_{2m-1}}] + [v_{i_{2m}}]) = \sum_{m=1}^N (v_{\ell_j}] + [v_{\ell_j}]) = 0,
$$
because $1+1 = 0$ in $\mathbb{Z}_2$. Otherwise, if $M = 2N+1$ for a nonnegative integer $N$, we have that
$$
[v_{i_1} + \cdots + v_{i_M}] = [v_{i_M}] + \sum_{m=1}^N ([v_{2i_m-1}] + [v_{2i_m}]) = [v_{i_M}] = [v_{\ell_j}]. 
$$
Thus, $[v_{i_1}], \dots, [v_{i_J}]$ forms a basis for $Z_0(K_n(s))/\big(B_0(K_n(t)) \cap Z_0(K_n(s))\big)$ and $\beta_{0,n}^{s,t} = J$, where $J$ is the number of $A_i$ that are non-empty. For each non-empty $A_\ell$, $\ell = 1, \dots, L$ there exists a unique $j$ and $i$ such that for $j-1$, and $j+i$ we have that $X_{j-1, n} \wedge X_{j+i, n} > t$, $X_{j,n} \vee \cdots \vee X_{j+i-1,n} \leq t$ and $X_{k,n} \leq s$ for some $k$ in the set $\{j, \dots, j+i-1\}$. Thus, \eqref{e:beta0st} follows. 

For the case when $t = \infty$, if there is any $k = 1, \dots, n$ such that $X_{k, n} \leq s$, then it follows from a simple argument---as used to demonstrate we have a spanning set in the finite $t$ case---that $[v_k]$ spans $Z_0(K_n(s))/\big(B_0(K_n(\infty)) \cap Z_0(K_n(s))\big)$.
\end{proof}

We now turn our attention to the proof of Lemma~\ref{l:conv_det}.

\begin{proof}[Proof of Lemma~\ref{l:conv_det}]
We will adapt the proof of Theorem A.2 from \cite{hiraoka2018}. 
First, it is clear that $\mathcal{R}$ is closed under finite intersections, so we have satisfied the first condition of Theorem 2.4 in \cite{billingsley} (i.e. that $\mathcal{R}$ is a $\pi$-system). It is also evident that $\Delta$ is separable. Now, for any $z \in \Delta$, $\epsilon > 0$ if we denote 
\[
\mathcal{R}_{z, \epsilon} := \{R \in \mathcal{R}: z \in R^{\circ} \subset R \subset B(z, \epsilon)\}, 
\]
then the class of boundaries $\partial \mathcal{R}_{z, \epsilon}$ contains uncountably many disjoint sets, regardless of if $z = (s, \infty)$ or $(s, t)$, where $t < \infty$ (in the former case $R^{\circ} = (s_1, s_2) \times (t_1, \infty]$). Thus $\mathcal{R}$ is a convergence-determining class by Theorem 2.4 of \cite{billingsley}. 

For the final part of the proof, let us fix a probability measure $\mu$ and choose an open set $U \subset \Delta$. Note that for every $z \in U$, there is an $\epsilon > 0$ such that $B(z, \epsilon) \subset U$. By the first part of this proof, for each of these $B(z, \epsilon)$ there exists a set $R_z \equiv R^U_z \in \mathcal{R}_{z, \epsilon}$ such that $\mu(\partial R_z) = 0$ and hence we have 
\[
U = \bigcup_{z \in U} R_z =  \bigcup_{z \in U} R^{\circ}_z,
\]
and $U$ is the union of sets with $\mu$-null boundaries. Since $\Delta$ is separable, there exists a countable subcover $\{R^U_{z_i}\}_{i=1}^{\infty}$ of $U$. Also, there exists a countable basis $\{U_j\}_{j=1}^{\infty}$ of $\Delta$. Hence, if we denote $R_{i,j} := R^{U_j}_{z_i}$ then
\[
U_j = \bigcup_{i=1}^{\infty} R_{i,j} = \bigcup_{i=1}^{\infty} R^{\circ}_{i,j}.
\]
If we let $\mathcal{R}_\mu$ be the class of finite intersections of the sets $R_{i,j}$. As the boundary of an intersection is a subset of the union of the boundaries, each element of $\mathcal{R}_\mu$ has a $\mu$-null boundary. Furthermore, every open set in $\Delta$ is the countable union of elements of $\mathcal{R}_\mu$. Hence, we apply Theorem 2.2 in \cite{billingsley} and the result holds. 

%

\end{proof}

Important to our SLLN was the existence of a certain limiting probability measure, whose existence and uniqueness we now tackle. 

\begin{lemma}\label{l:xi0}
The set function $\bar{\xi}_0$ on $\mathcal{R}$ defined at \eqref{e:setfunc} extends uniquely to a probability measure $\xi_0$ on the measure space $(\Delta, \mathcal{B}(\Delta))$.
\end{lemma}
\begin{proof}
To prove that $\bar{\xi}_0$ extends uniquely to a probability measure $\xi_0$, we will first demonstrate that the sequence of (uniformly bounded) measures 
$$
\mu_n := \frac{\E [\xi_{0,n} ]}{n\P(X_2 < X_1 \wedge X_3)}, \quad n = 1, 2, \dots
$$
is \emph{tight} in the sense that for every $\epsilon > 0$ there is a compact $K \subset \Delta$ such that $\sup_n \mu_n(\Delta \setminus K) \leq \epsilon$. If $(\mu_n)_{n \geq 1}$ is tight then for every subsequence $(\mu_{n_i})_{i \geq 1}$ there is a further subsequence $(\mu_{n_{i_j}})_{j \geq 1}$ such that there is some measure $\mu$ (not necessarily with $\mu(\Delta) = 1$) such that 
$$
\mu_{n_{i_j}}(A) \to \mu(A), \quad j \to \infty,
$$
for all Borel $A \in \mathcal{B}(\Delta)$ such that $\mu(\partial A) = 0$ (Lemma 4.4 in \citealp{kallenberg}). To prove that $(\mu_n)_{n \geq 1}$ is tight, note that the set $K_{s,t} = \{ (b,d) \in \Delta: s \leq b \leq d \leq t \}$ is compact and its complement $\Delta\setminus K_{s,t}$ satisfies
\begin{equation} \label{e:Kst}
\Delta\setminus K_{s,t} = \{(b,d) \in \Delta: b < s\} \cup \{(b,d) \in \Delta: d > t\},
\end{equation}
where we denote the first set on the righthand side of \eqref{e:Kst} as $B_s$  and the second as $D_t$. It follows that 
\begin{align*}
\frac{\E[ \xi_{0,n}(\Delta\setminus K_{s,t})]}{{n\P(X_2 < X_1 \wedge X_3)}} &=  \frac{\E[ \xi_{0,n}(B_s) ] +\E[ \xi_{0,n}(D_t) ]}{n\P(X_2 < X_1 \wedge X_3)} \\
&\leq \frac{\E[ \sum_{i=1}^n \ind{X_i < s} ] + \E[ \sum_{i=1}^n \ind{X_i > t}]}{n\P(X_2 < X_1 \wedge X_3)} \\
&= \frac{\P(X_i < s) + \P(X_i > t)}{\P(X_2 < X_1 \wedge X_3)}.
\end{align*}
As each $X_i$ are random variables and by definition finite almost surely, we may take $s$ small enough and $t$ large enough so that $\P(X_i < s) + \P(X_i > t) \leq \epsilon\P(X_2 < X_1 \wedge X_3)$. Thus, $(\mu_n)_{n \geq 1}$ is tight and for any subsequence $(n_i)_{i \geq 1}$ there is a further subsequence $(n_{i_j})_{j \geq 1}$ and a measure $\mu$ on the measure space $(\Delta, \mathcal{B}(\Delta))$ such that $\mu_{n_{i_j}}(A) \to \mu(A)$ for all $A$ with $\mu(\partial A) = 0$. However, the bounded convergence theorem implies that 
$$
\frac{\E [\xi_{0,{n_{i_j}}}(R) ]}{j\P(X_2 < X_1 \wedge X_3)} \to \bar{\xi}_0(R), 
$$
for each $R \in \mathcal{R}$. We have also established that
$$
\frac{\E [\xi_{0,n_{i_j}}(\Delta) ]}{j\P(X_2 < X_1 \wedge X_3)} \to 1,
$$
so that $\mu$ is a probability measure. Therefore, any subsequence of $(\mu_n)_{n \geq 1}$ has a subsequence that converges weakly to some $\mu$, but as $\mathcal{R}$ is a convergence-determining class for weak convergence, then any two probability measures on $(\Delta, \mathcal{B}(\Delta))$ which agree on $\mathcal{R}$ are identical. By the above argument and Theorem~2.6 in \cite{billingsley} we can conclude that there exists a unique probability measure $\xi_0$ such that $\mu_n \Rightarrow \xi_0$ as $n \to \infty$ which crucially satisfies $\xi_0(R) = \bar{\xi}_0(R)$ for $R \in \mathcal{R}$. 
\end{proof}

For the proof of our central limit theorem, we will employ Theorem 2.1 from \cite{neumann2013}, which establishes a CLT for potentially nonstationary weakly dependent triangular arrays. As mentioned at the beginning of Section~\ref{s:clt}, it is sufficient to show that 
\begin{equation} \label{e:cwbeta}
n^{-1/2} \sum_{l=1}^m a_l \bigg( \beta_{0,n}^{s_l, t_l} - \E[ \beta_{0,n}^{s_l, t_l} ] \bigg),
\end{equation}
converges to a Gaussian distribution for each $a_1, \dots, a_m \in \R$, to establish our desired convergence. Recall that at \eqref{e:ymjst} we defined the indicator (Bernoulli) random variable $Y^{m}_{j,n}(s,t)$ and on the following line we noticed that 
$$
\beta_{0,n}^{s,t} = \sum_{j=1}^n Y^{n-j+1}_{j,n}(s,t),
$$
so that \eqref{e:cwbeta} is equal to 
$$
n^{-1/2} \sum_{j=1}^n \sum_{l=1}^m a_l \bigg( Y_{j,n}^{n-j+1}(s_l, t_l) - \E[ Y_{j,n}^{n-j+1}(s_l, t_l) ] \bigg). 
$$
For the proof the CLT it is convenient for us to establish first a CLT for a truncated version of the persistent Betti numbers---as was done in the proof of the  Betti number CLT for the critical regime in the geometric setting, in Theorem 4.1 of \cite{owada_thomas}. Define first
$$
\beta_{0,n, K}^{s,t} = \sum_{j=1}^n Y^{(n-j+1) \wedge K}_{j,n}(s,t)
$$
Therefore, if we define 
$$
W_{j,n} := n^{-1/2}\sum_{l=1}^m a_l \bigg( Y_{j,n}^{(n-j+1) \wedge K}(s_l, t_l) - \E[ Y_{j,n}^{(n-j+1) \wedge K}(s_l, t_l) ] \bigg),
$$
establishing Theorem~\ref{t:bst_clt} amounts to establishing a CLT for $\sum_{j=1}^n W_{j,n}$ for each $K$ then showing that the difference between $\beta_{0,n}^{s,t}$ and $\beta_{0,n, K}^{s,t}$ disappears in probability. We will now quote the theorem which we will use to establish this. 


\begin{theorem}[Theorem~2.1 in \citealp{neumann2013}]\label{t:tri_clt}
Suppose that $(W_{j,n})_{j=1}^n$ with $n \in \N$ is a triangular array of random variables with $\E[W_{j,n}] = 0$ for all $j, n$ and $\sup_n \sum_{j=1}^n \E[W_{j,n}^2] \leq M$ for some $M < \infty$. Suppose further that
\begin{equation} \label{e:neu1}
\Var\bigg( \sum_{j=1}^n W_{j, n} \bigg) \to \sigma^2, \quad n \to \infty,
\end{equation}
for some $\sigma^2 \geq 0$, and that for every $\epsilon > 0$ we have
\begin{equation} \label{e:neu2}
\sum_{k=1}^n \E[W_{j,n}^2 \ind{ |W_{j,n} | > \epsilon } ] \to 0, \quad n \to \infty.
\end{equation}
Furthermore, assume that there exists a summable sequence of $\theta_r$, $r \in \N$, such that for all $q \in N$ and indices $1 \leq u_1 < u_2 < \cdots < u_q + r = v_1 \leq v_2 \leq n$, the following upper bounds for covariances hold true: 
\begin{equation} \label{e:neu3}
\Big | \mathrm{Cov}\big( g(W_{u_1, n}, \dots, W_{u_q, n})W_{u_q, n}, W_{v_1, n}\big) \Big | \leq \theta_r \big( \E[W_{u_q,n}^2] + \E[W_{v_1,n}^2] + n^{-1} \big)
\end{equation}
and
\begin{equation} \label{e:neu4}
\Big | \mathrm{Cov}\big( g(W_{u_1, n}, \dots, W_{u_q, n}), W_{v_1, n}W_{v_2, n}\big) \Big | \leq \theta_r \big( \E[W_{v_1,n}^2] + \E[W_{v_2,n}^2] + n^{-1}  \big)
\end{equation}
for all measurable $g: \R^q \to \R$ with $\sup_{x \in \R^q} |g(x)| \leq 1$. Then
$$
 \sum_{j=1}^n W_{j, n} \Rightarrow N(0, \sigma^2), \quad n \to \infty.
$$
\end{theorem}

\begin{proof}[Proof of Theorem~\ref{t:bst_clt}]
The finite-dimensional CLT proof for $\beta_{0,n,K}^{s,t}$ follows by checking that the conditions of Theorem~\ref{t:tri_clt} hold for our setup. First, we notice that 
\begin{align*}
&W_{j,n}^2 = n^{-1}\sum_{l_1 = 1}^m \sum_{l_2 = 1}^m a_{l_1}a_{l_2}\bigg( Y_{j,n}^{(n-j+1) \wedge K}(s_{l_1}, t_{l_1}) - \E[ Y_{j,n}^{(n-j+1) \wedge K}(s_{l_1}, t_{l_1}) ] \bigg) \\
&\phantom{W_{j,n}^2 = n^{-1}\sum_{l_1 = 1}^m \sum_{l_2 = 1}^m a_{l_1}a_{l_2}} \times \bigg( Y_{j,n}^{(n-j+1) \wedge K}(s_{l_2}, t_{l_2}) - \E[ Y_{j,n}^{(n-j+1) \wedge K}(s_{l_2}, t_{l_2}) ] \bigg),
\end{align*}
so that 
$$
\E[W_{j,n}^2]  = n^{-1}\sum_{l_1 = 1}^m \sum_{l_2 = 1}^m a_{l_1}a_{l_2} \Cov\Big( Y_{j,n}^{(n-j+1) \wedge K}(s_{l_1}, t_{l_1}), Y_{j,n}^{(n-j+1) \wedge K}(s_{l_2}, t_{l_2})\Big),
$$
which is bounded above by 
$$
n^{-1} \bigg( \sum_{l=1}^m |a_l| \sqrt{\Var(Y_{j,n}^{(n-j+1) \wedge K}(s_{l}, t_{l})}) \bigg)^2 \leq Mn^{-1},
$$
for $M := (\sum_l |a_l|)^2$ by the inequalities $|\Cov(X,Y)| \leq \sqrt{\Var(X)\Var(Y)}$ and $\Var(\mathbf{1}_A) \leq \P(A) \leq 1$. Thus, the condition $\sup_n \sum_{j=1}^n \E[W_{j,n}^2] \leq M < \infty$ holds. If we note that 
$$
\Var\bigg( \sum_{j=1}^n W_{j, n} \bigg) = n^{-1} \sum_{l_1=1}^m \sum_{l_2=1}^m a_{l_1}a_{l_2}\Cov\Big( \beta_{0,n, K}^{s_{l_1}, t_{l_1}}, \beta_{0,n, K}^{s_{l_2}, t_{l_2}} \Big), 
$$
then $\Var\big( \sum_{j=1}^n W_{j, n} \big)$ converges to some limit $\sigma^2$ via arguments analogous to and much simpler than those of Proposition~\ref{p:cov}. Thus, \eqref{e:neu1} is satisfied. If we use the triangle inequality, we can see that for each $j$
$$
|W_{j,n}| \leq 2n^{-1/2}\sum_{l=1}^m |a_l| = 2(Mn)^{-1/2}
$$
using the trivial indicator random variable bound $\mathbf{1}_A \leq 1$. Therefore when $n \geq 4(M/\epsilon)^2$ we have that $\ind{ |W_{j,n}| > \epsilon } = 0$ so that \eqref{e:neu2} holds as well. To finish the proof, we must show that \eqref{e:neu3} and \eqref{e:neu4} hold in Theorem~\ref{t:tri_clt} above. For both situations, we can ignore the case for $r \leq K + 1$, as we can set $\theta_r$ arbitrarily large in this case to get the required bounds in this case. Therefore, suppose that $r > K+1$, so that $W_{u_q, n}$ only depends on indices up to $u_q + K$ and $W_{v_1, n}$ only depends on indices starting at $v_1 - 1 = u_q + r - 1 > u_q + K$.

We will only demonstrate \eqref{e:neu3}, as \eqref{e:neu4} follows by a similar, simpler argument. For a fixed set of indices $u_1, \dots, u_q$ and fixed $n$ let us denote $G := g(W_{u_1, n}, \dots, W_{u_q, n})$. By the bilinearity of covariance, it will suffice to establish the required bounds in \eqref{e:neu3} for a single summand in
\begin{align}
n^{-1}\sum_{l_1 = 1}^m \sum_{l_2 = 1}^m a_{l_1}a_{l_2} \Cov\Big( G\big\{Y_{u_q,n}^{(n-u_q+1) \wedge K}(s_{l_1}, t_{l_1}) - \E[Y_{u_q,n}^{(n-u_q+1) \wedge K}(s_{l_1}, t_{l_1})] \big\},& \notag \\
\phantom{n^{-1}\sum_{l_1 = 1}^m \sum_{l_2 = 1}^m a_{l_1}a_{l_2} \Cov\Big(} Y_{v_1,n}^{(n-v_1+1) \wedge K}(s_{l_2}, t_{l_2})-\E[ Y_{v_1,n}^{(n-v_1+1) \wedge K}(s_{l_2}, t_{l_2})] \Big)&. \label{e:neu_cov}
\end{align}
provided that such a bound is uniform in $l_1, l_2$. It can be shown that the covariance term in \eqref{e:neu_cov} is equal to 
\begin{align}
&\Cov\Big(GY_{u_q,n}^{(n-u_q+1) \wedge K}(s_{l_1}, t_{l_1}), Y_{v_1,n}^{(n-v_1+1) \wedge K}(s_{l_2}, t_{l_2})\Big) \notag \\
&\qquad \qquad \qquad- \E[Y_{u_q,n}^{(n-u_q+1) \wedge K}(s_{l_1}, t_{l_1})]\Cov(G, Y_{v_1,n}^{(n-v_1+1) \wedge K}(s_{l_2}, t_{l_2})), \label{e:abs_neu}
\end{align}
and the absolute value of \eqref{e:abs_neu} can be bounded above by 
$$
|\Cov\Big(GY_{u_q,n}^{(n-u_q+1) \wedge K}(s_{l_1}, t_{l_1}), Y_{v_1,n}^{(n-v_1+1) \wedge K}(s_{l_2}, t_{l_2})\Big)| + |\Cov(G, Y_{v_1,n}^{(n-v_1+1) \wedge K}(s_{l_2}, t_{l_2}))|. 
$$
Because $GY_{u_q,n}^{(n-u_q+1) \wedge K}(s_{l_1}, t_{l_1}) = g^*(W_{u_1, n}, \dots, W_{u_q, n})$, for $g^*$ measurable and $\sup_x |g^*(x)| \leq 1$, the required bound will follow provided we find a suitable bound for the quantity $ |\Cov(G, Y_{v_1,n}^{(n-v_1+1) \wedge K}(s_l, t_l))|$. By definition of $\rho$-mixing and the trivial bound $\Var(X) \leq E[X^2]$ we have
\begin{align*}
|\Cov(G, Y_{v_1,n}^{(n-v_1+1) \wedge K}(s_l, t_l))| &\leq \rho_{\X}(r-1-K)\sqrt{\Var(G)}\sqrt{\Var(Y_{v_1,n}^{(n-v_1+1) \wedge K}(s_l, t_l))} \\
&\leq \rho_{\X}(r-1-K).
\end{align*}
By assumption, $\sum_{r > K+1} \rho_{\X}(r-1-K) < \infty$ so that \eqref{e:neu3} is established. As alluded to earlier, the proof for \eqref{e:neu4} follows in exactly the same way, hence
\begin{equation} \label{e:betti_trunc_clt}
n^{-1/2} \sum_{l=1}^m a_l \bigg( \beta_{0,n, K}^{s_l, t_l} - \E[ \beta_{0,n, K}^{s_l, t_l} ] \bigg) \Rightarrow N(0, \sigma_K^2), \quad n \to \infty
\end{equation}
for all $K \in \N$. As the dominated convergence assumption holds true in Proposition~\ref{p:cov}, it is straightforward to see that $\sigma_K^2 \to \sigma^2$ as $K \to \infty$, where $\sigma^2$ is the limiting variance of \eqref{e:cwbeta}. Hence $N(0, \sigma_K^2) \Rightarrow N(0, \sigma^2)$ as $K \to \infty$ as well (using L{\'e}vy's continuity theorem, for example). Theorem 3.2 in \cite{billingsley} will yield the rest if we can show that 
$$
\lim_{K \to \infty} \limsup_{n \to \infty} \P( | Z_{n,K} - Z_n | \geq \epsilon) = 0,
$$
where $Z_n$ is the sum of persistent Betti numbers in \eqref{e:cwbeta} and $Z_{n,K}$ is the $K$-truncated version on the left-hand side of \eqref{e:betti_trunc_clt}. An application of Chebyshev's inequality and the covariance inequality yields
\begin{align*}
\P( | Z_n - Z_{n,K} | \geq \epsilon) &\leq \frac{\E \big | Z_n - Z_{n,K} \big |^2}{\epsilon^2} \notag \\
&= \frac{1}{\epsilon^2 n}\Var\bigg( \sum_{l=1}^m a_l \Big[ \beta_{0,n}^{s_l, t_l} - \beta_{0,n, K}^{s_l, t_l} \Big] \bigg)  \notag\\
&\leq  \frac{1}{\epsilon^2}\bigg( \sum_{l=1}^m a_l \sqrt{n^{-1}\Var\big( \beta_{0,n}^{s_l, t_l} - \beta_{0,n, K}^{s_l, t_l} \big)}\Bigg)^2. \label{e:bill_cond}
\end{align*}
The quantity $n^{-1}\Var\big( \beta_{0,n}^{s_l, t_l} - \beta_{0,n, K}^{s_l, t_l} \big)$ converges to a limit defined by the terms \eqref{e:covsum1}, \eqref{e:covsum2}, and \eqref{e:covsum3} below with the restriction that $i_1, i_2 > K$. As each of the sums in \eqref{e:covsum1}, \eqref{e:covsum2}, and \eqref{e:covsum3} are absolutely convergent, their restriction with $i_1, i_2 > K$ tends to 0 as $K \to \infty$, and the CLT follows.

Finally, for any $A \in \mathcal{B}(\tilde{\Delta})$ 
\[
\xi_{0,n}(A) = \tilde{\xi}_{0,n}(A)
\]
so that if each coordinate of $R_l$ is in $\R$, then $R_l \in \mathcal{B}(\tilde{\Delta})$ and the result is proved for the restricted persistence diagram as well.
\end{proof}
We finish this section with a proof of the limiting covariance seen in Proposition~\ref{p:cov}, which we will break into a few lemmas. \\

\noindent\emph{Proof of Proposition~\ref{p:cov}}. 
Let us define
\begin{equation*}
C^n_{i,j}(s,t) := \mathbf{1} \bigg\{  \bigvee_{k=j}^{j+i-1} X_{k,n} \leq t, \bigwedge_{k=j}^{j+i-1} X_{k,n} \leq s \bigg \}\ind{X_{j-1, n} \wedge X_{j+i, n} > t},
\end{equation*}
where $C_{i,j}(s,t) \equiv C^{\infty}_{i,j}(s,t)$ is analogously defined for the entire sequence $\X$. Therefore we have
\[
\beta_{0,n}^{s, t} = \sum_{j=1}^n \sum_{i=1}^{n-j+1} C^n_{i,j}(s,t) = \sum_{i=1}^n \sum_{j=1}^{n-i+1} C^n_{i,j}(s,t). \\
\]
Thus, it follows that
\begin{align}
&\mathrm{Cov}\Big(\beta_{0,n}^{s_1, t_1}, \beta_{0,n}^{s_2, t_2}\Big) \notag \\
&\qquad = \E\big[ \beta_{0,n}^{s_1, t_1} \beta_{0,n}^{s_2, t_2} \big] - \E[\beta_{0,n}^{s_1, t_1}] \E[\beta_{0,n}^{s_2, t_2}] \notag \\
&\qquad = \sum_{i_1, i_2} \sum_{j_1, j_2} \E[\cconest \cctwost] - \E[\cconest]\E[\cctwost], \label{e:ccov_st}
\end{align}
where $i_1, i_2 = 1, \dots, n$ with $j_1 = 1, \dots, n-i_1+1$, and $j_2 = 1, \dots, n-i_2+1$. We may then break \eqref{e:ccov_st} into
\begin{align}
&\sum_{i_1, i_2} \sum_{j=1}^{n-i_1\vee i_2+1} \E[C^{n}_{i_1,j}(s_1, t_1)C^{n}_{i_2,j}(s_2, t_2)] - \E[C^{n}_{i_1,j}(s_1, t_1)]\E[C^{n}_{i_2,j}(s_2, t_2)] \notag \\
&\qquad + \sum_{i_1, i_2} \sum_{k=1}^{n-i_2} \sum_{j=1}^{n-i_1 \vee (i_2+k)+1} \E[C^{n}_{i_1,j}(s_1, t_1)C^{n}_{i_2,j+k}(s_2, t_2)] - \E[C^{n}_{i_1,j}(s_1, t_1)]\E[C^{n}_{i_2,j+k}(s_2, t_2)] \notag \\
&\qquad + \sum_{i_1, i_2} \sum_{k=1}^{n-i_1} \sum_{j=1}^{n-i_2\vee (i_1+k)+1} \E[C^{n}_{i_1,j+k}(s_1, t_1)C^{n}_{i_2,j}(s_2, t_2)] - \E[C^{n}_{i_1,j+k}(s_1, t_1)]\E[C^{n}_{i_2,j}(s_2, t_2)] \label{e:bigcov1}.
\end{align}
For now, we will exclude the boundary terms from each sum---which use $X_{0,n}$ and $X_{n+1,n}$. We will treat the boundary terms later. The nonboundary terms of the expression \eqref{e:bigcov1} can thus be simplified based on the assumed stationarity of $\X_n$ to be
\begin{align}
&\sum_{i_1, i_2} (n-i_1\vee i_2-1) \Big(\E[C^n_{i_1,2}(s_1, t_1)C^n_{i_2,2}(s_2, t_2)] - \E[C^n_{i_1,2}(s_1, t_1)]\E[C^n_{i_2,2}(s_2, t_2)] \Big) \notag \\
&\ + \sum_{i_1, i_2} \sum_{k=1}^{n-i_2} (n-i_1 \vee (i_2+k)-1) \Big(\E[C^n_{i_1,2}(s_1, t_1)C^n_{i_2,2+k}(s_2, t_2)] - \E[C^n_{i_1,2}(s_1, t_1)]\E[C^n_{i_2,2+k}(s_2, t_2)]\Big) \notag \\
&\ + \sum_{i_1, i_2} \sum_{k=1}^{n-i_1} (n-i_2\vee (i_1+k)-1) \Big(\E[C^n_{i_1,2+k}(s_1, t_1)C^n_{i_2,2}(s_2, t_2)] - \E[C^n_{i_1,2+k}(s_1, t_1)]\E[C^n_{i_2,2}(s_2, t_2)]\Big) \label{e:bigcov2}.
\end{align}
Dividing by $n$, we may express the first term in \eqref{e:bigcov2} as 
\begin{equation} \label{e:term1_1}
\sum_{i_1=1}^{\infty} \sum_{i_2=1}^{\infty} (1-(i_1\vee i_2-1)/n)_+ \Big(\E[C^n_{i_1,2}(s_1, t_1)C^n_{i_2,2}(s_2, t_2)] - \E[C^n_{i_1,2}(s_1, t_1)]\E[C^n_{i_2,2}(s_2, t_2)] \Big).
\end{equation}
Assuming we can show that
\[
\sum_{i_1=1}^{\infty} \sum_{i_2=1}^{\infty} \Big| \E[C_{i_1,2}(s_1, t_1)C_{i_2,2}(s_2, t_2)] - \E[C_{i_1,2}(s_1, t_1)]\E[C_{i_2,2}(s_2, t_2)]\Big| < \infty,
\]
where we drop the superscript $n$ as mentioned at the start of Section~\ref{s:clt}, then \eqref{e:term1_1} will converge to 
\begin{equation}\label{e:covsum1}
\sum_{i_1=1}^{\infty} \sum_{i_2=1}^{\infty} \E[C_{i_1,2}(s_1, t_1)C_{i_2,2}(s_2, t_2)] - \E[C_{i_1,2}(s_1, t_1)]\E[C_{i_2,2}(s_2, t_2)].
\end{equation}
Similarly, we will get limits of 
\begin{equation}\label{e:covsum2}
\sum_{i_1=1}^{\infty} \sum_{i_2=1}^{\infty} \sum_{k=1}^{\infty} \E[C_{i_1,2}(s_1, t_1)C_{i_2,2+k}(s_2, t_2)] - \E[C_{i_1,2}(s_1, t_1)]\E[C_{i_2,2+k}(s_2, t_2)],
\end{equation}
and
\begin{equation}\label{e:covsum3}
\sum_{i_1=1}^{\infty} \sum_{i_2=1}^{\infty} \sum_{k=1}^{\infty} \E[C_{i_1,2+k}(s_1, t_1)C_{i_2,2}(s_2, t_2)] - \E[C_{i_1,2+k}(s_1, t_1)]\E[C_{i_2,2}(s_2, t_2)],
\end{equation}
for the second and third terms in \eqref{e:bigcov2}, provided the dominated convergence assumption holds for each of these cases. In fact, these three sums comprise the limit of the covariance. However, to establish that, we must ensure that the ``boundary terms'' vanish, which we do in Lemma~\ref{l:vanish_terms}. A useful fact will aid in the proof of the covariance limit above and the lemma below.

\begin{lemma}\label{l:cov_disappear}
Fix $k \geq 0$. Suppose that $i_2 + k > i_1$ and $k \leq i_1$, then for any values of $t_1, t_2$ we have
\[
C^{n}_{i_1,j}(s_1, t_1)C^{n}_{i_2,j+k}(s_2, t_2) = 0.
\]
Analogously, if $i_1 + k > i_2$ and $k \leq i_2$, then for any values of $t_1, t_2$ we have
\[
C^{n}_{i_1,j+k}(s_1, t_1)C^{n}_{i_2,j}(s_2, t_2) = 0.
\]
\end{lemma}

\begin{proof}
Note that if $i_2 + k > i_1$ and $k \leq i_1$, then it must be the case that there exists indices $l, l'$ such that if $C_{i_1,j}(s_1,t_1)C_{i_2,j+k}(s_2,t_2)=1$ then 
\[
X_l \leq t_1, X_l > t_2 \text{ and } X_l' > t_1, X_l' \leq t_2,
\]
a contradiction because $t_1 > t_2$ and $t_2 > t_1$ cannot simultaneously hold---even if $t_1=t_2$. The proof for the second case follows by the same argument.
\end{proof}

\begin{lemma} \label{l:vanish_terms}
If $\X$ is a $\rho$-mixing stationary stochastic process that is max-root summable then the boundary terms in \eqref{e:bigcov1} are $o(n)$ as $n \to \infty$.
\end{lemma}

\begin{proof}
The boundary terms \eqref{e:bigcov1} comprise those terms in the first sum that satisfy $j = 1$ or $j+(i_1 \vee i_2) = n+1$, the terms in the second sum satisfying $j=1$ or $j+i_1 \vee( i_2+k) = n+1$, and the terms in the third sum satisfying $j=1$, or $j+(i_1+k)\vee i_2 = n+1$. Thus, the boundary terms can be represented as 
\begin{align}
&\sum_{i_1, i_2} \Cov\big( C^n_{i_1,1}(s_1,t_1), C^n_{i_2,1}(s_2,t_2) \big) + \Cov\big( C^n_{i_1,n-i_1 \vee i_2+1}(s_1,t_1), C^n_{i_2,n-i_1 \vee i_2+1}(s_2,t_2) \big) \notag  \\
& +  \sum_{i_1, i_2} \sum_{k=1}^{n-i_2}   \Cov\big( C^n_{i_1,1}(s_1,t_1), C^n_{i_2,1+k}(s_2,t_2) \big) \notag \\
&\qquad \qquad \qquad  + \Cov\big( C^n_{i_1,n-i_1 \vee (i_2+k)+1}(s_1,t_1), C^n_{i_2,n-(i_1-k) \vee i_2+1}(s_2,t_2) \big) \notag \\
& +  \sum_{i_1, i_2} \sum_{k=1}^{n-i_1}  \Cov\big( C^n_{i_1,1+k}(s_1,t_1), C^n_{i_2,1}(s_2,t_2) \big) \notag \\
&\qquad \qquad \qquad + \Cov\big( C^n_{i_1,n-i_1\vee (i_2-k)+1}(s_1,t_1), C^n_{i_2,n-(i_1+k) \vee i_2+1}(s_2,t_2) \big) \label{e:3bd}.
\end{align}
We may bound the absolute value of the first sum in \eqref{e:3bd} by
\begin{align*}
&\sum_{i_1, i_2} \sqrt{\Var(C^n_{i_1,1}(s_1,t_1))}\sqrt{\Var( C^n_{i_2,1}(s_2,t_2) )} \\
&\phantom{\sum_{i_1, i_2}}\qquad\qquad + \sqrt{\Var( C^n_{i_1,n-(i_1 \vee i_2)+1}(s_1,t_1))}\sqrt{\Var(C^n_{i_2,n-i_1 \vee i_2+1}(s_2,t_2))} \\
&\leq 2 \sum_{i_1, i_2} \sqrt{\P(X_1 \leq t_1, \dots, X_{i_1} \leq t_1)}\sqrt{\P(X_1 \leq t_2, \dots, X_{i_2} \leq t_2)} < \infty,
\end{align*}
and thus $o(n)$---where we use the inequalities $|\Cov(X,Y)| \leq \sqrt{\Var(X)\Var(Y)}$, $\Var(\mathbf{1}_A) \leq \P(A)$, and the fact that $\X$ is max-root summable. We now will finish the proof by showing that the second sum in \eqref{e:3bd} is $o(n)$ as well. That the third sum in \eqref{e:3bd} is $o(n)$ follows by an essentially symmetric proof. We may bound the absolute value of the second sum in \eqref{e:3bd} by 
\begin{align}
2&\sum_{i_1, i_2} \sum_{k=i_1+2}^{n-i_2} \rho_\X(k-i_1-1)  \sqrt{\Var(C^n_{i_1,1}(s_1,t_1))}\sqrt{\Var( C^n_{i_2,1}(s_2,t_2) )}\notag \\
+&\sum_{i_1, i_2} \sum_{k=1}^{i_1+1}  \Big| \Cov\big( C^n_{i_1,1}(s_1,t_1), C^n_{i_2,1+k}(s_2,t_2) \big) \Big| \notag \\
&\qquad \qquad \qquad + \Big| \Cov\big( C^n_{i_1,n-i_1 \vee (i_2+k)+1}(s_1,t_1), C^n_{i_2,n-(i_1-k) \vee i_2+1}(s_2,t_2) \big) \Big|. \label{e:bd_2bd}
\end{align}
The first sum in \eqref{e:bd_2bd} follows from the definition of $\rho$-mixing and the fact that $F(s_i) > 0$ and $F(t_i) < 1$. Dividing the aforementioned first sum by $n$ we see that 
\begin{align*}
&2n^{-1}\sum_{i_1, i_2} \sum_{k=i_1+2}^{n-i_2} \rho_\X(k-i_1-1) \sqrt{\Var(C^n_{i_1,1}(s_1,t_1))}\sqrt{\Var( C^n_{i_2,1}(s_2,t_2) )} \\
&\qquad \leq 2n^{-1}\sum_{k=1}^n \rho_\X(k) \sum_{i_1, i_2} \sqrt{\P(X_1 \leq t_1, \dots, X_{i_1} \leq t_1)}\sqrt{\P(X_1 \leq t_2, \dots, X_{i_2} \leq t_2)} \\
\end{align*}
which tends to $0$ as $n \to \infty$ by max-root summability and the fact that $\rho_\X(k) \to 0$. The second sum in \eqref{e:bd_2bd} is a little more delicate. Before continuing, note that $\big| k \in \N: k > i_1 - i_2, \, k \leq i_1\big| = i_1 \wedge i_2 \leq i_1i_2$ when both terms are at least 1. Hence, Lemma~\ref{l:cov_disappear} implies that the second sum in \eqref{e:bd_2bd} equals
\begin{align}
&\sum_{i_2=1}^n \sum_{i_1=1}^{n} \sum_{k=1+(i_1-i_2)_+}^{i_1} \E[C^n_{i_1,1}(s_1,t_1)] \E[C^n_{i_2,1+k}(s_2,t_2)] \notag \\
&\phantom{\quad + \sum_{i_2=1}^n \sum_{i_1=i_2+1}^n \sum_{k=1}^{i_1-i_2}} \qquad + \E[C^n_{i_1,n-i_1 \vee (i_2+k)+1}(s_1,t_1)] \E[C^n_{i_2, n-(i_1-k) \vee i_2+1}(s_2,t_2)] \notag \\
&\quad + \sum_{i_2=1}^n \sum_{i_1=1}^{n-i_2-1}  \Big| \Cov\big( C^n_{i_1,1}(s_1,t_1), C^n_{i_2,i_1+2}(s_2,t_2) \big) \Big| + \Big| \Cov\big( C^n_{i_1,n-(i_1+i_2)}(s_1,t_1), C^n_{i_2,n-i_2+1}(s_2,t_2) \big) \Big| \notag \\
&\quad + \sum_{i_2=1}^n \sum_{i_1=i_2+1}^n \sum_{k=1}^{i_1-i_2} \bigg(  \Big| \Cov\big( C^n_{i_1,1}(s_1,t_1), C^n_{i_2,1+k}(s_2,t_2) \big) \Big| \label{e:bd_sum_final}. \\
&\phantom{\quad + \sum_{i_2=1}^n \sum_{i_1=i_2+1}^n \sum_{k=1}^{i_1-i_2}} \qquad + \Big| \Cov\big( C^n_{i_1,n-i_1+1}(s_1,t_1), C^n_{i_2,n-i_1+k+1}(s_2,t_2) \big) \Big|\bigg). \notag
\end{align}
We may bound the first term in \eqref{e:bd_sum_final} 
\begin{align*}
&2\sum_{i_2=1}^n \sum_{i_1=1}^{n} i_1i_2 \P(X_1 \leq t_1, \dots, X_{i_1} \leq t_1)\P(X_1 \leq t_2, \dots, X_{i_2} \leq t_2) \\
&\qquad = 2 \sum_{i_1=1}^{n} i_1 \P(X_1 \leq t_1, \dots, X_{i_1} \leq t_1) \sum_{i_2=1}^n i_2 \P(X_1 \leq t_2, \dots, X_{i_2} \leq t_2) \\
&\qquad = o(n)
\end{align*}
by the max-root summability condition. Furthermore, we can bound the second sum in \eqref{e:bd_sum_final} by 
\[
2 \sum_{i_1=1}^{n} \sqrt{\P(X_1 \leq t_1, \dots, X_{i_1} \leq t_1)} \sum_{i_2=1}^{n} \sqrt{\P(X_1 \leq t_2, \dots, X_{i_2} \leq t_2)} = o(n), 
\]
using the covariance inequality $\Cov(X,Y) \leq \sqrt{\Var(X)\Var(Y)}$, and again using the max-root summability condition. Finally, we bound the third sum in \eqref{e:bd_sum_final} by 
\begin{align*}
& 2 \sum_{i_2=1}^n \sum_{i_1 = i_2+1}^n (i_1-i_2) \sqrt{\P(X_1 \leq t_1, \dots, X_{i_1} \leq t_1)} \sqrt{\P(X_1 \leq t_2, \dots, X_{i_2} \leq t_2)} \\
&\leq 2 \sum_{i_1 = 1}^n i_1 \sqrt{\P(X_1 \leq t_1, \dots, X_{i_1} \leq t_1)} \sum_{i_2=1}^n \sqrt{\P(X_1 \leq t_2, \dots, X_{i_2} \leq t_2)} \\ 
&= o(n),
\end{align*}
by a final application of the max-root summability condition.
\end{proof}

Having shown that the boundary terms vanish under our conditions, it will suffice to show the dominated convergence condition for the terms in \eqref{e:bigcov2} divided by $n$, which will then tend to the sums of \eqref{e:covsum1}, \eqref{e:covsum2}, and \eqref{e:covsum3} respectively. First, we divide each term by $n$ and see that the first covariance term with absolute summands is bounded above (using again the usual covariance inequalities) by
\begin{align*}
&\sum_{i_1, i_2} \sqrt{\P(X_1 \leq t_1, \dots, X_{i_1} \leq t_1)} \sqrt{\P(X_1 \leq t_2, \dots, X_{i_2} \leq t_2)} \\ 
&\qquad =\sum_{i_1} \sqrt{\P(X_1 \leq t_1, \dots, X_{i_1} \leq t_1)} \sum_{i_2} \sqrt{\P(X_1 \leq t_2, \dots, X_{i_2} \leq t_2)} < \infty,
\end{align*}
by applying max-root summability for each sum. We now prove the dominated convergence assumption for the second sum (divided by $n$) in \eqref{e:bigcov2}, as the third sum follows an analogous proof. This procedure yields an upper bound of 
\begin{align}
& \sum_{i_1, i_2} \sum_{k=1}^{n-i_2} \Big| \Cov\big(C^{n}_{i_1,2}(s_1, t_1)C^{n}_{i_2,2+k}(s_2, t_2)\big) \Big| \notag \\
&\qquad \leq \sum_{i_1, i_2} \sum_{k=1}^{i_1+1} \Big| \Cov\big(C^{n}_{i_1,2}(s_1, t_1)C^{n}_{i_2,2+k}(s_2, t_2)\big) \Big| \notag \\
&\qquad+ \sum_{i_1, i_2} \sum_{k=i_1+2}^{n-i_2} \rho_{\X}(k-i_1-1) \sqrt{\Var(C^n_{i_1,2}(s_1,t_1))}\sqrt{\Var( C^n_{i_2,2}(s_2,t_2) )}. \label{e:cov2_dom}
\end{align}
The first sum in \eqref{e:cov2_dom} we may bound by 
\[
\sum_{i_1, i_2} (i_1+1) \sqrt{\P(X_1 \leq t_1, \dots, X_{i_1} \leq t_1)}\sqrt{\P(X_1 \leq t_2, \dots, X_{i_2} \leq t_2)} < \infty,
\]
by max-root summability of $\X$. The second sum in \eqref{e:cov2_dom} is bounded above by 
\begin{align*}
&\sum_{k=1}^{n} \rho_{\X}(k )\sum_{i_1, i_2} \sqrt{\Var(C^n_{i_1,1}(s_2,t_1))} \sqrt{\Var( C^n_{i_2,2}(s_2,t_2) )} \\
\leq &\sum_{k=1}^{\infty} \rho_{\X}(k ) \sum_{i_1, i_2} \sqrt{\P(X_1 \leq t_1, \dots, X_{i_1} \leq t_1)}\sqrt{\P(X_1 \leq t_2, \dots, X_{i_2} \leq t_2)}  < \infty.
\end{align*}
by assumption. 

As for the representation of $\lim_{n \to \infty} n^{-1}\mathrm{Cov}\Big(\beta_{0,n}^{s_1, t_1}, \beta_{0,n}^{s_2, t_2}\Big)$, we note that the sums \eqref{e:covsum1}, \eqref{e:covsum2}, and \eqref{e:covsum3} are all absolutely convergent, hence we may split the sums and apply the monotone convergence theorem to each, and recombine to get the stated representation. \hfill $\square$ \bigskip

\bibliographystyle{plainnat}
\bibliography{LimitDistributionsPHStationaryProcesses}

\end{document}